\documentclass[12pt]{article}
\usepackage{amsmath,amsfonts,amssymb}
\usepackage{alltt}

\usepackage{amsmath,amsfonts,ifthen,fullpage,enumerate}  
\usepackage{mathrsfs}
\usepackage{hyperref}
\usepackage{mathtools}
\usepackage{array}
\usepackage{tikz-cd}
\usetikzlibrary{arrows.meta}
\tikzset{>={Latex[width=2mm,length=2mm]}}

\usepackage{colortbl}

\usepackage{enumerate}

\usepackage{amsmath,amsfonts,amssymb}

\usepackage{makecell}  

\usepackage{graphicx}        

\usepackage{tikz-cd}

\usepackage{tcolorbox}
\def\greyboxvar#1#2
{\begin{tcolorbox} [boxrule=-5pt, boxsep=-0.1cm, colback=white!95!black, 
	width=#2] #1\end{tcolorbox}}

\def\greybox#1
{
	\begingroup
\setlength{\tabcolsep}{2pt} 
\renewcommand{\arraystretch}{1.3} 
	\begin{tabular}{c}
\cellcolor{gray!25}#1\end{tabular}
\endgroup
}

\usepackage[all,pdf]{xy}
\usepackage{lmodern,amssymb}   

\def\gcd{\mathop{\rm gcd}\nolimits}

\newcommand\divides{\mid}

\def\Re{\mathop{\rm Re}\nolimits}

\def\pmat#1{\begin{pmatrix}#1\end{pmatrix}}
\def\mat#1{\begin{matrix}#1\end{matrix}}
\def\question#1{{\bf Question: }#1}
\def\question#1{}

\def\cG{{\cal G}}

\def\cL{{\cal L}}
\def\cO{{\cal O}}
\def\cT{{\cal T}}

\def\R{\mathbb{R}}

\def\CC{\mathbb{C}}

\def\FF{\mathbb{F}}
\def\HH{\mathbb{H}}

\def\Cd{\C^d}
\def\Fd{\FF^d}

\def\Hd{\HH^d}

\def\Rd{\R^d}

\def\C{\mathbb{C}}

\newcommand{\RR}{\mathbb{R}}

\newtheorem{theorem}{Theorem}[section]
\newtheorem{corollary}{Corollary}[section]
\newtheorem{lemma}{Lemma}[section]
\newtheorem{example}{Example}[section]

\newtheorem{proposition}{Proposition}[section]

\newenvironment{proof}{{\noindent \it
Proof.}}{\hfill$\Box$\medskip}
\newif\ifdraft\def\draft{\drafttrue\hoffset=.8truecm\showlabeltrue
\def\comment##1{{\bf comment: ##1}}
\headline={\sevenrm \hfill \ifx\filenamed\undefined\jobname\else\filenamed\fi%
(.tex) (as of \ifx\updated\undefined???\else\updated\fi)
 \TeX'ed at {\hour\time\divide\hour by 60{}%
\minutes\hour\multiply\minutes by 60{}%
\advance\time by -\minutes
\the\hour:\ifnum\time<10{}0\fi\the\time\  on \today\hfill}}
}

\def\inpro#1{\langle#1\rangle}
\def\ip#1{\langle\kern-.28em\langle#1\rangle\kern-.28em\rangle_\nu}

\def\cC{{\cal C}}
\def\cD{{\cal D}}
\def\cO{{\cal O}}
\def\cT{{\cal T}}
\def\cI{{\cal I}}
\def\cH{{\cal H}}

\def\openR{{{\rm I}\kern-.16em {\rm R}}}

\def\Fd{\FF^d}
\let\ga\alpha

\let\gb\beta

\let\gd\delta

\let\gl\lambda

\let\gL\Lambda

\let\gs\sigma

\let\go\omega
\let\gO\Omega
\let\ga\alpha

\let\gb\beta
\let\gd\delta
\let\gs\sigma
\def\inpro#1{\langle#1\rangle}

\def\GL{\mathop{\it GL}\nolimits}

\def\Iff{\hskip1em\Longleftrightarrow\hskip1em}
\def\Implies{\hskip1em\Longrightarrow\hskip1em}

\def\formeq{\the\sectionno.\the\equationno}  
\def\elabel#1/#2/#3/{\global\advance\equationno by 1 %
\ifx#1\empty\else\emember#1%
\ifshowlabel\marginal{\string#1}\fi\fi%
\ifmmode\eqno{#3(\formeq#2)}\else#3\formeq#2\fi} 

\def\makeblanksquare#1#2{
\dimen0=#1pt\advance\dimen0 by -#2pt
      \vrule height#1pt width#2pt depth0pt\kern-#2pt
      \vrule height#1pt width#1pt depth-\dimen0 \kern-#1pt
      \vrule height#2pt width#1pt depth0pt \kern-#2pt
      \vrule height#1pt width#2pt depth0pt
}

\title{\bf 
The quaternionic systems of imprimitivity for the reflection groups of rank two
}

\author{Shayne Waldron\\ 
 \\
Department of Mathematics \\ University of Auckland\\
Private
Bag 92019, Auckland, New Zealand\\
e--mail: waldron@math.auckland.ac.nz}

\begin{document}

\maketitle 

\begin{abstract}

Given an explicit presentation of a reflection group of rank two 
(or any rank two group for that matter), we give a simple procedure for calculating
all its systems of imprimitivity, when viewed as a matrix group
over the quaternions.
This is applied to all the reflection groups, in particular the quaternionic 
reflection groups, thereby unifying a number of results and ideas in the literature.
For example, a primitive complex reflection group of rank two 
has either uncountably many quaternionic
	systems of imprimitivity ($3$ cases) or none ($16$ cases).


\end{abstract}

\bigskip
\vfill

\noindent {\bf Key Words:}
systems of imprimitivity,
irreducible groups of rank two,
imprimitive quaternionic reflection groups,
reflection systems,
binary polyhedral groups,
dicyclic groups,
finite collineation groups.

\bigskip
\noindent {\bf AMS (MOS) Subject Classifications:}
primary
15B33, \ifdraft (Matrices over special rings (quaternions, finite fields, etc.) \else\fi
20F55, \ifdraft	Reflection and Coxeter groups (group-theoretic aspects) \else\fi
20G20, \ifdraft Linear algebraic groups over the reals, the complexes, the quaternions \else\fi
51F15, \ifdraft	Reflection groups, reflection geometries \else\fi
\quad
secondary
20C25, \ifdraft Projective representations and multipliers \else\fi
51M20. \ifdraft (Polyhedra and polytopes; regular figures, division of spaces [See also 51F15]) \else\fi

\vskip .5 truecm
\hrule
\newpage

\section{Introduction}

The (irreducible) reflection groups, i.e., finite groups generated by reflections,
have been classified into those which are real \cite{C34}, complex \cite{ST54}
and quaternionic \cite{C80}.
A {\bf reflection} on $\Rd$, $\Cd$ or $\Hd$ is a linear map $r$ which pointwise 
fixes a subspace of dimension $d-1$, and has finite order, i.e., satisfies
$r v= v\xi$ for some nonzero vector $v$ (called a root of the reflection)
and a scalar $\xi\ne1$ with $\xi^m=1$ ($m$ the order of $r$). 
For real reflections $\xi=-1$, and for complex reflections any $\gl\ne1$ in 
the cyclic group $\inpro{\go}$ 
generated by the $m$-th root of unity $\go=\xi$
gives a reflection which is a power of $r$. 
In the quaternionic setting, which is of ongoing interest, see
\cite{BST23}, \cite{S23}, \cite{DZ24}, \cite{BW25},
the group $\inpro{\go}$ is replaced by a finite (multiplicative) 
subgroup of $\HH^*=\HH\setminus\{0\}$.   

We will use $\FF=\RR,\CC,\HH$ when we can 
treat the three cases simultaneously
(we seek to unify the theory as much as is possible). 
The subgroups of $G\subset\GL(\FF^d)$, such as the reflection groups, 
are classified up to a change of basis (which preserves reflections),
i.e., conjugation in $\GL(\FF^d)$. 
We write $\cong_\FF$ for conjugacy in $\GL(\FF^d)$. 
See \cite{Z97}, \cite{CS03}, \cite{V21} for general facts about groups of matrices over
the quaternions $\HH$.

If $G\subset\GL(\FF^d)$ is
group, then we say that it (or its linear action on $\Fd$) has a 
{\bf system of imprimitivity} $V_1,\ldots,V_m$ of $m\ge2$ nonzero subspaces if
the action of $G$ permutes the $V_j$'s and $\Fd=V_1\oplus\cdots\oplus V_m$ 
(internal direct sum). In this case, $G$ is said to be {\bf imprimitive}, and otherwise
it is {\bf primitive}. If $G$ is irreducible, i.e., $\{gv\}_{g\in G}$ spans $\Fd$ for
every $v\ne0$, and $G$ is a reflection group, 
then any system of imprimitivity must have $\dim_\FF(V_j)=1$, and
so the matrices of $G$ can be represented as monomial matrices (each row or column 
has exactly one nonzero entry) by choosing a basis from the system of imprimitivity.
It may be (as we will see for reflection groups) that a subgroup of
$\GL(\Rd)$ or $\GL(\Cd)$ is primitive, but is imprimitive when viewed
as a group of matrices over the larger field (division algebra)
$\CC,\HH$ or $\HH$, respectively.

The general question considered in this paper is: when does an 
irreducible group $G$ 
have a system of imprimitivity when viewed as matrices 
over $\RR,\CC,\HH$? This boils down to determining
$$ \hbox{\em Which (if any) changes of bases give $G$ as a monomial group of matrices?} $$
Here we consider the rank two groups (groups of $2\times2$ matrices over $\RR,\CC,\HH$),
which considerably simplifies the problem.
The main results are the {\em explicit} calculation of {\em all} the systems of imprimitivity
(including the quaternionic ones) for the 
real, complex and quaternionic reflection groups 
(Theorem \ref{realreflectsystems}, 
Theorems \ref{imprimiitivecomplexsys}, \ref{complexprimitivesys}, 
Theorems \ref{G(n,a,b,r)sys}, \ref{TOIsys}
and 
Table \ref{RealSystems-table}, Table \ref{ComplexSystems-table},
Table \ref{QuaternionicSystems-table}). The overall picture is as follows:
\begin{itemize}
\item There is just one imprimitive real reflection group $D_4=G(2,1,2)\cong_\CC G(4,4,2)$. 
		It has two real systems of
	primitivity, three complex systems of primitivity, 
and infinitely many quaternionic systems of primitivity.
\item The primitive real reflection groups have one complex system of imprimitivity,
	and infinitely many quaternionic systems of imprimitivity.
\item There are $16$ real and complex reflection groups with no quaternionic systems 
of primitivity (all of them complex), and all others have uncountably many systems.
\item The imprimitive complex reflection group $G(4,2,2)$ has three
	complex systems of imprimitivity, and the others have just one.
All of them have infinitely many quaternionic systems of imprimitivity.
\item An imprimitive quaternionic reflection group has either one, two, three,
five (one case $G(2,1,2,1)$) systems of imprimitivity, 
		or infinitely many (the three cases where it is conjugate to the primitive complex reflection groups $G_{12},G_{13},G_{22}$).
\end{itemize}

We have three basic observations (which we will exploit repeatedly):

\begin{enumerate}
\item Enlarging a group decreases the systems of imprimitivity,
\end{enumerate}
i.e., any system of imprimitivity for $G$ is also a
system of imprimitivity for any subgroup $H\subset G$ 
($H$ may have more, e.g., consider $H=1$).
Moreover, the complex systems of imprimitivity 
for a group of real or complex matrices are unchanged if nonzero scalar matrices are added,
and so the systems of imprimitivity (over $\RR$ or $\CC$), depend only on the 
collineation group (i.e., the group of matrices up to scalar multiplication). 
Since each nonscalar matrix in $\GL(\CC^2)$ can be multiplied by two
scalars (the inverses of its eigenvalues) to obtain a reflection,
each collineation group is associated with a (maximal) reflection group
(see \cite{W26}), and so 
the systems of imprimitivity in $\CC^2$ for the finite rank two subgroups of $\GL(\CC^2)$
	are given by those of the corresponding (maximal) complex reflection groups.
The systems of imprimitivity over $\HH$ are seen to depend on the scalar matrices in $G$,
i.e., matrix groups over $\CC$ giving the same collineation group may have
different quaternionic systems of imprimitivity
(see Example \ref{collineationsystems}).

\begin{enumerate}
\setcounter{enumi}{1}
\item The systems of imprimitivity for a group $G$ depend only on its generators.
\item The finite group $G$ may be taken to unitary for the standard inner product 
$$\inpro{v,w}:=\sum_j\overline{v_j}w_j, \qquad v,w\in\Fd, $$
so that its systems of imprimitivity are orthogonal, 
and any possible change of basis matrix $U$ for a system of
	imprimitivity can be chosen to be unitary.
\end{enumerate}
Henceforth, we will assume that all groups are unitary, which simplifies 
finding their systems of primitivity (which are orthogonal and have
a unitary change of basis matrix), and $\Hd$ is a right vector space, 
so that linear maps (matrices) are applied on the left.

\section{How to find systems of imprimitivity}

We now consider all the possible orthogonal systems (the candidates for systems
of imprimitivity). There is of course the standard orthonormal basis $\{e_1,e_2\}$.
Any other orthogonal system must have a vector which is nonzero in the first coordinate
(otherwise it would give the standard basis), and so after scaling it is given 
by the equal-norm orthogonal vectors
$$ \pmat{1\cr q}, \quad \pmat{-\overline{q}\cr 1}, \qquad q\in\HH, \quad q\ne0. $$
We consider how many times an orthogonal system is given by 
a vector $(1,q)$, $q\in\HH$.

\begin{proposition}
\label{doublecountinprop}
Every orthogonal system for $\FF^2$ is given by a vector
	$$ \pmat{1\cr q}, \qquad q\in\FF, \quad |q|\le1, $$
with different values of $q$ giving different orthogonal systems, 
unless $|q|=1$, in which case the same orthogonal system is 
	given by $(1,\pm q)$ (i.e., these are counted twice above).
\end{proposition}

\begin{proof}
For $q\ne0$, we have
$$ \pmat{-\overline{q}\cr 1} = \pmat{1\cr -q/|q|^2}(-\overline{q}), $$
so that if $|q|>1$, the orthogonal system given by $(1,q)$ is also given 
by $(1,q')$, where
	$$ q'=-{1\over|q|^2} q, \qquad |q'|={1\over|q|}<1, $$
so that all orthogonal systems are obtained with the restriction $|q|\le1$. 

Suppose that $(1,q)$ and $(1,q')$ give the same orthogonal system, i.e., 
	one of  
	$$ \inpro{\pmat{1\cr q},\pmat{1\cr q'}}= 1+\overline{q}q'=0, \qquad
	\inpro{\pmat{1\cr q},\pmat{-\overline{q'}\cr 1}}= -\overline{q'}+\overline{q}=0 
\Iff q'=q, $$
holds. The systems could only be the same for $q'\ne q$, $|q|,|q'|\le1$ 
if the first holds, i.e.,
$$ \overline{q}q'=-1 \Implies |q|=|q'|=1, \quad q'=-q, $$
and we have the claimed double indexing.
\end{proof}

A unitary change of basis matrix for the orthogonal
system for $(1,q)$ is given by
\begin{equation}
\label{Uqdefn}
U={1\over\sqrt{1+|q|^2}} \pmat{1&\overline{q}\cr q&-1}.
\end{equation}
We note that the columns of $U$ have been scaled so that $U^2=I$, i.e., $U^*=U$.
All the possible unitary change of basis matrices are obtained by multiplying the
columns of $U$ by unit scalars. We will sometimes do this to obtain nice
formulas, see, e.g., (\ref{Uscale}).

We give a condition for a unitary matrix $g$ to have a 
system of imprimitivity given by $(1,q)$, i.e., for $U^{-1}gU=U^*gU$ to
be a monomial matrix.

\begin{lemma}
\label{abcdq-lemma}
The orthogonal system given by $(1,q)$, $q\in\HH$, is a system of imprimitivity for 
a unitary group of matrices $G\subset\GL(\FF^2)$ if and only if for every 
generator
$$ g=\pmat{a&b\cr c&d}, $$
in a set of generators for $G$, one of the following two equations holds
\begin{equation}
\label{abcdqeqns} 
a +bq +\overline{q}c +\overline{q}dq=0, \qquad qa +qbq -c -dq=0.
\end{equation}
\end{lemma}

\begin{proof}
Let $U$ be the change of basis matrix (\ref{Uqdefn}) for the orthogonal
system given by $(1,q)$. Then the matrix representation of $g$ in this basis is
\begin{align*}
U^{-1}gU &={1\over1+|q|^2}
\pmat{1&\overline{q}\cr q&-1} \pmat{a&b\cr c& d} \pmat{1&\overline{q}\cr q&-1} \cr
&= {1\over1+|q|^2} \pmat{a +bq +\overline{q}c +\overline{q}dq &
a\overline{q} -b +\overline{q}c\overline{q} -\overline{q}d \cr 
qa +qbq -c -dq &
	qa\overline{q} -qb -c\overline{q} +d\overline{q} },
\end{align*}
which is a monomial matrix if and only if one of the entries of the first column 
is zero (this implies the same for the second column since $U^{-1}gU$ 
is unitary), i.e., one of the equations in 
	(\ref{abcdqeqns}) holds.
\end{proof}

We observe that for $(1,0)$ to give a system of imprimitity (the standard basis), 
the condition (\ref{abcdqeqns}) for $q=0$ reduces to $a=0$ or $c=0$, i.e., 
that $g$ is monomial.

\section{The systems of imprimitivity of the real reflection groups}

To illustrate the calculations and results to come 
(for the complex and quaternionic reflection groups), 
we now use Lemma \ref{abcdq-lemma} to find the systems of imprimitivity for the
irreducible real reflection groups of rank two.
These are the {\bf dihedral groups}
\begin{equation}
\label{Dihedraldefn}
D_n = \inpro{R,S}, \quad n\ge3,  \qquad 
R=\pmat{\cos{2\pi\over n}&-\sin{2\pi\over n}\cr\sin{2\pi\over n}&\cos{2\pi\over n}}, 
\quad S=\pmat{1&0\cr0&-1},
\end{equation}
generated by a rotation $R$ by ${2\pi\over n}$ and a reflection $S$ in the $x$-axis,
which is the symmetry group of the regular $n$-gon.
Writing $c_n=\cos{2\pi\over n}$, $s_n=\sin{2\pi\over n}$, the conditions
of (\ref{abcdqeqns}) for $R$ and $S$ are
$$ c_n -s_n q +\overline{q}s_n +\overline{q}c_nq=0, \qquad 
qc_n -qs_n q -s_n -c_n q=0, $$
i.e.,
\begin{equation}
\label{DihedeqnI}
c_n(1+|q|^2)+s_n(\overline{q}-q) =0, \qquad s_n(q^2+1) =0, 
\end{equation}
and 
\begin{equation}
\label{DihedeqnII}
1 -\overline{q}q=1-|q|^2=0, \qquad q +q=2q=0.
\end{equation}

Taking $q=0$ to satisfy the second equation in (\ref{DihedeqnII}) reduces
(\ref{DihedeqnI}) to
$$ c_n =0, \qquad s_n=0, $$
one of which can hold only for $n=4$ ($c_4=0$).
Thus only $D_4$ has the standard basis as a system of imprimitivity
(it is a monomial group, as is seen by inspection).


Taking $|q|=1$, $q=a+bi+cj+dk\in\HH$, to satisfy the first equation in (\ref{DihedeqnII}), 
reduces (\ref{DihedeqnI}) to
$$ 2c_n-2s_n(bi+cj+dk) =0, \qquad s_n(q^2+1) =0 \Iff q^2=-1. $$
Since $s_n\ne0$, every dihedral group has a system of imprimitivity
given by $(1,q)$, where $q^2=-1$ (this is equivalent to $\Re(q)=0$ and $|q|=1$).
In particular $(1,i)$ gives a complex system of primitivity.
Since $c_n\ne0$, $n\ne4$, and $c_4=0$, there is second system of primitivity for
$D_4$ given by $(1,1)$. 

Thus we have found all the systems of primitivity for the real reflection groups.

\begin{theorem}
\label{realreflectsystems}
The systems of imprimitivity for the irreducible real
	reflection groups $D_n$, $n\ge3$, are given by $(1,q)$, $q\in\HH$, 
	in the following cases
\begin{enumerate}[(a)]
\item $(1,0)$, $(1,1)$ for $D_4$.
\item $(1,i)$ for $D_n$, $n\ge3$.
\item $(1,q)$, $|q|=1$, $q\in\HH\setminus\CC$, $\Re(q)=0$, for $D_n$, $n\ge3$.
\end{enumerate}
with the possible double countings described in Proposition \ref{doublecountinprop}.
\end{theorem}

\vskip-0.55truecm

\begin{table}[h]
\caption{The systems of imprimitivity for the real reflection groups $D_n$, $n\ge3$, of 
	(\ref{Dihedraldefn}). The condition $|q|=1$, $\Re(q)=0$ is equivalent to $q^2=-1$.
	In particular, we observe that $D_4$ is imprimitive, with two systems of
	imprimitivity, and the other groups are primitive.}
\label{RealSystems-table}
\begin{center}
\vskip-0.25truecm
\begin{tabular}{ |  >{$}l<{$} | >{$}c<{$} | >{$}l<{$} | >{$}l<{$} | l | }
\hline
&&&&\\[-0.3cm]
	G & \hbox{real} & \hbox{complex} & \hbox{quaternionic} & \hbox{comment} \\[0.1cm]
\hline
&&&&\\[-0.3cm]
	D_n,\, n\ne4 & & (1,i) & (1,q),\, |q|=1,\, \Re(q)=0 & primitive, $\cong_\CC G(n,n,2)$ \\ [3pt]
	D_4 & (1,0),\ (1,1) & (1,i) & (1,q),\, |q|=1,\, \Re(q)=0 & imprimitive, $\cong_\CC G(4,4,2)$ \\[3pt]
\hline
\end{tabular}
\end{center}
\end{table}

Since $(1,i)$ gives a (complex) system of imprimitivity for all the real reflection 
groups $D_n$ of (\ref{Dihedraldefn}), 
they may be conjugated by the $U$ of (\ref{Uqdefn}) for $q=i$ to obtain 
a complex imprimitive reflection group. With $\go=e^{2\pi i\over n}$, 
the generators for this group 
are
\begin{align}
\label{dihedralUbasis}
U^{-1}RU &={1\over2}\pmat{1&-i\cr-i&1}\pmat{c_n&-s_n\cr s_n&c_n}
\pmat{1&i\cr i&1} 
=\pmat{c_n-is_n&0\cr0&c_n+is_n} = \pmat{\overline{\go}&0\cr0&\go}, \cr
	U^{-1}SU &={1\over2} \pmat{1&-i\cr-i&1} \pmat{1&0\cr0&-1} \pmat{1&i\cr i&1} 
= \pmat{0&-i\cr i&0}, 
\end{align}
which is the imprimitive ``dihedral'' group $G(n,n,2)$ in the 
Shephard-Todd classification of complex reflection groups 
(see Example \ref{leftrightmult}). 
The Shephard-Todd group $G(2,1,2)$ is precisely our $D_4$, which explains
the isomorphism
$$ G(2,1,2) \cong_\CC G(4,4,2), $$
which is the only isomorphism between the Shephard-Todd groups $G(n,p,d)$ for $d\ge2$
fixed, where $d$ is the rank of the group.

We now seek to do essentially the same calculations for the complex and quaternionic
reflection groups of rank two. Many of these are imprimitive, i.e., already in 
monomial form, and so in this case we are looking for {\it additional} systems of imprimitivity.

The classification of the imprimitive complex and quaternionic reflection groups of rank two
(and the real one for that matter) proceeds from the observation that the monomial
reflections have two types
\begin{equation}
\label{type1type2}
\pmat{0&b\cr \overline{b}&0}, \qquad \pmat{h&0\cr0&1}, \ \pmat{1&0\cr0&h}, \quad h\ne1, 
\end{equation}
which have orders $2$ and the order of $h$, respectively. The roots for these
reflections 
are $(1,-\overline{b})$, $e_1$, $e_2$, respectively.

The set $L$ of the $b$'s giving the reflections of the first type (for a given 
imprimitive
reflection group) are closed under the binary operation
$$ (a,b)\mapsto a\circ b:=ab^{-1}a, $$
and form what is called a {\bf reflection system} in \cite{W25}.
If the closure of a subset $\cL\subset L$ under the operation $\circ$ is $L$,
then 
we say that $\cL$ {\bf generates} the reflection system $L$.
If $1\in L$, then the multiplicative group $K$ generated by the reflection system $L$ is a
finite subgroup of $\HH^*$. 
We will enumerate the possible $K$ (all of which
give imprimitive reflection groups) as we go through the classification.
The group $K=\inpro{-1}\subset\RR$ gives the real reflection group $D_4$, 
$K=\inpro{\go}\subset\CC$, $\go=e^{2\pi i\over n}$, $n\ge3$, gives the 
complex reflection groups, and the binary tetrahedral, octahedral, icosahedral
and dihedral groups $K=\cT,\cO,\cI$ and $\cD_n$, $n\ge2$, the quaternionic ones
(see Table \ref{TOIDngroups-table}).
The set $H$ of the $h$'s giving the reflections of the second type together with $1$
is a normal subgroup of $K$. 

The imprimitive reflection group $G(K,L,H)$ is
defined to be the reflection group generated by the reflections of the types
(\ref{type1type2}) given by $b\in L$, $h\in H\setminus\{1\}$, as above.
Canonical choices of $(K,L,H)$ which give {\em all} the reflection groups {\em without
repetition} (groups being conjugate) are given in \cite{W25}
(see Table \ref{TOIDngroups-table} below).
A small set of generating reflections for these groups corresponding 
to subsets $\cL\subset L$ and $\cH\subset H$ are given, and we have
\begin{equation}
\label{G(K,L,H)defn}
G(K,L,H)=\cG(\cL,\cH):=\inpro{\bigl\{\pmat{0&\gb\cr\overline{\gb}&0}\bigr\}_{\gb\in\cL}\bigcup
\bigl\{\pmat{h&0\cr0&1}\bigr\}_{h\in\cH} }.
\end{equation}
It is enough to simply use the generators given in
Table \ref{TOIDngroups-table}. 
We observe that for these
\begin{enumerate}
\item $\cL$ is a generating set for the reflection system $L$ (which can be
	labelled to indicate $K$ and its number of elements). 
\item $\cL$ always contains $1$.
\item $\cH$ need not contain $1$, and $\cH=\{\}$ gives what is called the {\em base group},
	and larger sets $\cH$ give the {\em higher order groups} for the given 
		reflection system $L$.
\end{enumerate}
One advantage of (\ref{G(K,L,H)defn}) over the classification of \cite{C80}
for quaternionic reflection groups is that the groups are given explicitly 
with a small number of generating reflections. Another is that inclusions of the form
\begin{equation}
\label{GKLHsinclusion}
G(K_1,L_1,H_1)\subset G(K_2,L_2,H_2), \qquad
K_1\subset K_2,\ L_1\subset L_2,\ H_1\subset H_2,
\end{equation}
are readily apparent. It should also be noted that the classification of \cite{C80}
is not correct, it both over counts and under counts reflection groups
(see \cite{W25}, \cite{T25}).

\begin{table}[!h]
\caption{The imprimitive reflection groups $G=G_K(L,H)=\cG(\cL,\cH)$ obtained from the
reflection systems $L$ for $K=\cD_n$ ($n\ge2$, $[n,a,b,r]\in\gL_n$) and $K=\cT,\cO,\cI$
of \cite{W25}.
The base groups have $H=1$, 
and the $\cL$ given for the base group generates  
	$L$.
The only conjugate groups are $G_\cT(L_{12}^\cT,C_2)\cong_\HH G_\cO(L_{14}^\cO,1)$.
The number of reflections in $G$ is $|L|+2(|H|-1)$, e.g.,
$G(n,a,b,r)=G_{\cD_n}(L_{(a,b)}^{(n)},C_r)$, 
has
${2n\over a}+{2n\over b}+2(r-1)$ reflections.
	}
	\vskip0.3truecm
\label{TOIDngroups-table}
\begin{center}
\begin{tabular}{ |  >{$}l<{$} | >{$}l<{$} | >{$}l<{$} | >{$}l<{$} | >{$}l<{$} | >{$}l<{$} | >{$}l<{$} | >{$}l<{$} |}
\hline
	&&&&&\\[-0.3cm]
K & L & H & |G| & \cL & \cH \\[0.1cm]
\hline
&&&&&\\[-0.3cm]
\cD_n & \cD_n & \cD_n & 32n^2 & \{1,\go,j,\go j\} & \{ \go,j 	\} \\
\cD_n & \cD_n & \cD_{n/2} & 16n^2 & \{1,\go,j,\go j\} & \{ j 	\} \\
	\cD_n & L_{(a,b)}^{(n)} & C_r & 8nr & \{1,\go^a,j,\go^b j\}\, ([n,a,b,r]\in\gL_n^*) & \{ \go^{2n\over r} 	\} \\
&&&&&\\[-0.3cm]
\cT & \cT & \cT & 1152 & \{1,i\} & \{{1+i+j+k\over2}\} \\
\cT & \cT & Q_8 & 384 & \{1,i,j,{1+i+j+k\over2}\} & \{\} \\
\cT & L_{12}^\cT & C_2 & 96 & \{1,i,{1+i+j+k\over2}\} & \{-1\} \\
\cT & L_{12}^\cT & 1 & 48 & \{1,i,{1+i+j+k\over2}\} & \{\} \\
&&&&&\\[-0.3cm]
\cO & \cO & \cO & 4608 & \{1,{1+i+j+k\over2} \} & \{{1+i\over\sqrt{2}}\} \\
\cO & \cO & \cT & 2304 & \{1, {1+i\over\sqrt{2}},{1+j\over\sqrt{2}}, {1+i+j+k\over2} \} & \{\} \\
\cO & L_{32}^\cO & Q_8 & 768 & \{1,{1+i\over\sqrt{2}},j,{1+i+j+k\over2}\} & \{\} \\
\cO & L_{20}^\cO & C_2 & 192 & \{1,{1+i\over\sqrt{2}},{1+i+j+k\over2},{j-k\over\sqrt{2}}\} & \{\} \\
\cO & L_{18}^\cO & 1 & 96 & \{1,{1+i\over\sqrt{2}},{1+i+j+k\over2}\} & \{\} \\
\cO & L_{14}^\cO & 1 & 96 & \{1,i,{1+i+j+k\over2},{j-k\over\sqrt{2}}\} & \{\} \\
&&&&&\\[-0.3cm]
\cI & \cI & \cI & 28800 & \{1,i,{1+i+j+k\over2},{\tau+\gs i-j\over2}\} & \{ \} \\
\cI & L_{32}^\cI & C_2 & 480 & \{1,{1+i+j+k\over2},{\tau+\gs i-j\over2},{j-\tau i-\gs k\over2}\} & \{ \} \\
\cI & L_{30}^\cI & 1 & 240 & \{1,{1+i+j+k\over2},{\tau+\gs i-j\over2}\} & \{ \} \\
\cI & L_{20}^\cI & C_2 & 480 & \{1,i,{1+i+j+k\over2},{i+\gs j+\tau k\over2}\} & \{-1 \} \\
\cI & L_{20}^\cI & 1 & 240 & \{1,i,{1+i+j+k\over2},{i+\gs j+\tau k\over2}\} 
& \{ \} \\ [0.1cm]
\hline
\multicolumn{6}{l}{\hbox{\hskip5truecm\footnotesize $\go=e^{\pi i\over n}$, $\tau={1+\sqrt{5}\over2}$, $\gs=1-\tau$ }} 
\end{tabular}
\end{center}
\end{table}

\begin{example}
\label{leftrightmult}
Left or right multiplication of a reflection system by some unit scalar $x$ gives
another reflection system, e.g.,
	$$ (xa)\circ(xb) = xa(xb)^{-1}xa
	=xab^{-1}x^{-1}xa=-xab^{-1}a=x(a\circ b), $$
which is considered to be equivalent, since
\begin{equation}
\label{RefSys-equiv}
\pmat{x&0\cr 0&1}^{-1}M_{xb}\pmat{x&0\cr 0&1}
=\pmat{1&0\cr 0&x}M_{bx}\pmat{1&0\cr 0&x}^{-1}=M_b,
\qquad M_b:=\pmat{0&b\cr b^{-1}&0}.
\end{equation}
The reflections $S,SR,\ldots,SR^{n-1}$ of the dihedral group of (\ref{Dihedraldefn})
in the basis of (\ref{dihedralUbasis}) are
$$ U^{-1} SR^j U 
= \pmat{0&-i\cr i&0} \pmat{\overline{\go}&0\cr0&\go}^j
= \pmat{0&-i{\go}^j\cr(-i{\go}^j)^{-1}&0}, $$
giving the reflection system $\{-i,-i\go,\ldots,-i\go^{n-1}\}$, 
which is equivalent to $\{1,\go,\ldots,\go^{n-1}\}$ 
	(which has generating set $\{1,\go\}$), so 
	by (\ref{Gnp2defn}), we have that $D_n\cong_\CC G(n,n,2)$.
This can also be obtained directly by appropriately scaling the orthonormal basis, 
i.e.,
\begin{equation}
\label{Uscale}
V^{-1}SV=\pmat{0&1\cr1&0}, \quad V^{-1} SR V=\pmat{0&\go\cr\overline{\go}&0},
\qquad
V=U\pmat{-i&0\cr0&1}={1\over\sqrt{2}}\pmat{-i&-i\cr1&-1}.
\end{equation}
\end{example}

Since the generating reflections in (\ref{G(K,L,H)defn}) are 
monomial matrices, 
the Lemma \ref{abcdq-lemma} takes the following simplified form.

\begin{lemma}
\label{imprimitivitylemma}
The vector $(1,q)\in\HH^2$, $q\ne0$, gives 
     an additional system of
imprimitivity for the imprimitive reflection group $G=G(K,L,H) =\cG(\cL,\cH)$
if and only if
\begin{enumerate}[(i)]
	\item $\Re(\gb q)=0$ or $q=\pm\overline{\gb}$, for all $\gb\in\cL$. (base group)
	\item $h=-|q|^2$, for all $h\in\cH$, $h\ne1$. (higher order groups)
\end{enumerate}
In particular, $G$ can have additional systems of imprimitivity
only when all its reflections have order two.
\end{lemma}

\begin{proof}
For the generators
$$ g=\pmat{0&\gb\cr\overline{\gb}&0}, \quad\gb\in\cL, \qquad
	g=\pmat{h&0\cr0&1},\quad j\in\cH, $$
the condition (\ref{abcdqeqns}) in Lemma \ref{abcdq-lemma} reduces to
$$ \gb q +\overline{q}\overline{\gb} =0 \Iff \Re(\gb q)=0, 
\qquad -q\gb q +\overline{\gb} =0 \Iff (\gb q)^2=1 \Iff q=\pm\overline{\gb}, $$
$$ h +\overline{q}q=0 \Iff h=-|q|^2, \qquad -qh +q=0 \Iff h=1, $$
	respectively, and we obtain the conditions (i) and (ii).

If $G$ had a reflection of order $m\ge3$, then it would be given by some
	$h$ of order $m$, and (ii) would be false.
\end{proof}

We observe that for a given $\gb=\gb_1+\gb_2 i+\gb_3 j+\gb_4 k\in\cL$,
the first condition of (i) gives a homogeneous
linear equation in the coefficients of $q=a+bi+cj+dk$, i.e.,
\begin{equation}
\label{abcdeqn}
\Re(\gb q)=\gb_1 a-\gb_2 b-\gb_3 c-\gb_4 d=0.
\end{equation}
Further, if $\gb q\in\CC$, then we have the equivalences
\begin{equation}
\label{realgbqcdn}
\Re(\gb q)=0 \Iff \gb q=\pm|q|i \Iff (\gb q)^2=-|q|^2.
\end{equation}

We observe that only case where there could be a $q$ with $|q|\ne1$, 
and hence a continuous infinite family of systems of imprimitivity is when
$$ q\not\in\{\pm\overline{\gb}:\gb\in\cL\}, \qquad H=1. $$

\section{The complex reflection groups and their systems of imprimitivity }

We now consider the 
complex reflection groups (see \cite{BMR95}, \cite{LT09}).


Let $\go=e^{2\pi i\over n}$.
The (irreducible) imprimitive complex reflection groups of rank two are given by 
the cyclic groups $K=C_n=\inpro{\go}$. They are
\begin{equation}
\label{Gnp2defn}
G(n,p,2) =\cG(\{1,\go\},\{\go^p\}), \quad p\divides n, \qquad n\ge3, \quad
	(n,p,2)\ne(4,4,2).
\end{equation}
We exclude $n=2$, which gives the real reflection groups
$G(2,1,2)=D_4$ (as already discussed) and 
$G(2,2,2)$ (which is not irreducible).
We will also see that $G(4,4,2)$ is conjugate to the real imprimitive reflection 
group $G(2,1,2)=D_4$.

\begin{theorem}
\label{imprimiitivecomplexsys}
The imprimitive complex reflection groups $G(n,p,2)$, $p\divides n$, $n\ge3$,
as defined by (\ref{Gnp2defn}, have in addition to the standard basis,
systems of imprimitivity given by $(1,q)$, $q\in\HH$, $q\ne0$, in
the following cases
\begin{enumerate}[(a)]
\item $(1,1)$, $(1,i)$ for $G(4,4,2)$ and $G(4,2,2)$. (the last group contains the first).
\item $(1,zj)$, $z\in\CC$ ($z\ne0$), for $G(n,n,2)$.
\item $(1,zj)$, $z\in\CC$, $|z|=1$, for $G(n,{n\over2},2)$, when $n$ is even.
\end{enumerate}
The group $G(4,4,2)$ is conjugate to the real reflection group $D_4=G(2,1,2)$, 
and we note
the inclusion 
	$$ \hbox{$G(n,n,2)\subset G(n,{n\over2},2).$} $$
\end{theorem}

\begin{proof}
We apply Lemma \ref{imprimitivitylemma} to the generating reflections of (\ref{Gnp2defn})
for $G(n,p,2)$, i.e.,
$$  \cL=\{1,\go\}, \qquad \cH=\{\go^p\}.  $$

The condition (ii) for $\cH=\{\go^p\}$   
holds if and only if $h=\go^p=1$, i.e., $p=n$, or $h=\go^p=-|q|^2$, i.e., $p={n\over2}$
and $|h|=1$.

The conditions of (i) for $\gb\in\cL=\{1,\go\}$ are
$$ \Re(q)=0, \qquad q=\pm1, $$
$$ \Re(\go q)=0, \qquad q= \pm \overline{\go}. $$
If we take $q=\pm1$, then we must have $\Re(\go)=0$, 
i.e., $n=4$, and there is a system of imprimitivity given by $(1,1)$ 
for $G(4,4,2)$ and $G(4,2,2)$.
We now seek a system with $q=\pm\overline{\go}\ne\pm1$, since $\Re(q)=0$, 
we must have $n=4$, and we obtain $(1,i)$ for $G(4,4,2)$, $G(4,2,2)$.
Finally, we seek a system with 
$$ \Re(q)=0, \qquad \Re(\go q)=0, $$
i.e., by (\ref{abcdeqn}), a quaternion $q=a+bi+cj+dk$ with
$$ a=0, \qquad \bigl(\cos{2\pi\over n}\bigr)a-\bigl(\sin{2\pi\over n}\bigr)b=0
\Implies a=b=0 \quad\hbox{(since $n\ge3$}). $$
The $q$ obtained in this way can be written
$q=cj+dk=zj$, $z\in\CC$.

The conjugacy $G(4,4,2)\cong_\CC G(2,1,2)$ is given in the Example \ref{G442conjexample} 
to follow, and the inclusion follows immediately from (\ref{GKLHsinclusion}).
\end{proof}

\begin{example} 
\label{G442conjexample}
For $n=4$, we consider $G(4,p,2)=G(4,4,2),G(4,2,2)$.
Since $\go=i$, these groups are generated by the reflections
$$ \pmat{0&1\cr1&0},\pmat{0&i\cr-i&0}, \quad \pmat{-1&0\cr0&1}\ 
\hbox{(for $p=2$)}. $$
In the system of imprimitivity given by $\{(1,1),(-1,1)\}$ these are
$$ \pmat{1&0\cr0&-1}, \pmat{0&i\cr-i&0}, \quad \pmat{0&1\cr1&0}, $$
and in the system $\{(1,i),(i,1)\}$ they are
$$ \pmat{0&1\cr1&0},\pmat{-1&0\cr0&1}, \quad \pmat{0&-i\cr i&0}. $$
From the last, we conclude that $G(4,4,2)$ is conjugate to $G(2,1,2)=D_4$. 
\end{example}

\begin{example} 
We consider the mechanics of the system of imprimitivity given by
\begin{equation}
\label{zjmechanics}
U={1\over\sqrt{1+|z|^2}}\pmat{1&zj\cr zj&1}, \quad z\in\CC,  \quad
U^{-1}=U^*={1\over\sqrt{1+|z|^2}}\pmat{1&-zj\cr -zj&1}.
\end{equation}
In this system, the generators 
$$ g= \pmat{0&1\cr1&0},\pmat{0&\go\cr\overline{\go}&0}, \quad
\pmat{\go^p&0\cr0&1}, $$
for $G(n,p,2)$ become
	$$ U^{-1}gU= \pmat{0&1\cr1&0},\pmat{0&\go\cr\overline{\go}&0}, \quad
{1\over1+|z|^2}\pmat{\go^p+|z|^2&(\go^p-1)zj\cr
(1-\overline{\go}^p)zj&\overline{\go}^p|z|^2+1}, $$
with the last clearly monomial if $\go^p=1$ or $|z|=1$, $\go^p=-1$.
Moreover, we have
\begin{equation} 
\label{j-iso}
U^{-1}\pmat{0&j\cr-j&0}U 
= {1\over{1+|z|^2}}\pmat{-z-\overline{z} &(1-z^2)j\cr (z^2-1)j &z+\overline{z}} 
= {1\over{1+|z|^2}}\pmat{-2\Re(z) &(1-z^2)j\cr (z^2-1)j &2\Re(z)},
\end{equation}
which is monomial and real if and only if $z=\pm1$.
\end{example}

The $19$ (irreducible) primitive complex reflection groups
of rank two are denoted by
$$ G_4,\ldots,G_7 \ \hbox{(tetrahedral)}, \quad
 G_8,\ldots,G_{15} \ \hbox{(octahedral)}, \quad
 G_{16},\ldots,G_{22} \ \hbox{(icosahedral)}. $$
Will use the explicit unitary generators of \cite{W26} for these groups, which 
satisfy certain inclusions (see Figure \ref{GjInclusions}).
Only $G_{12}$, $G_{13}$ and $G_{22}$ will turn out to have quaternionic systems of imprimitivity,
and these have the following generators.
Let
$$ F:= {1\over\sqrt{2}}\pmat{1&1\cr1&-1}, \quad
R:= \pmat{1&0\cr0&-i}, \quad 
Z= e^{2\pi i\over24} RF =\hbox{${1+\sqrt{3}+(\sqrt{3}-1)i\over4}$}\pmat{1&1\cr-i& i}, $$
$$ A:=R^2=\pmat{1&0\cr0&-1}, \quad
M:=\hbox{$(-{\tau\over4}+{\sqrt{1-\tau^2/4}\over2} i)$}
\pmat{-\tau+i&\gs\cr -\gs&-\tau-i}, $$
with $\tau={1+\sqrt{5}\over2}$, $\gs=1-\tau={1-\sqrt{5}\over2}$.
Then
\begin{equation}
\label{G12G13G22defn}
G_{12} = \inpro{F,F^Z,F^{Z^2}}, \qquad
G_{13} = \inpro{F,F^Z,R^2}, \qquad
G_{22} = \inpro{A,A^Z,A^M}, 
\end{equation}
where
$$ 
F^Z=ZFZ^{-1}= {1\over\sqrt{2}}\pmat{1&i\cr-i&-1}, \qquad
F^{Z^2}=   {1\over\sqrt{2}}\pmat{0&1+i\cr1-i&0}, $$
$$ A^Z=\pmat{0&i\cr-i&0}, \quad
A^M=\pmat{{\tau\over2}&-{1\over2}-{\gs\over2}i\cr -{1\over2}+{\gs\over2}i & -{\tau\over2}}.
 $$

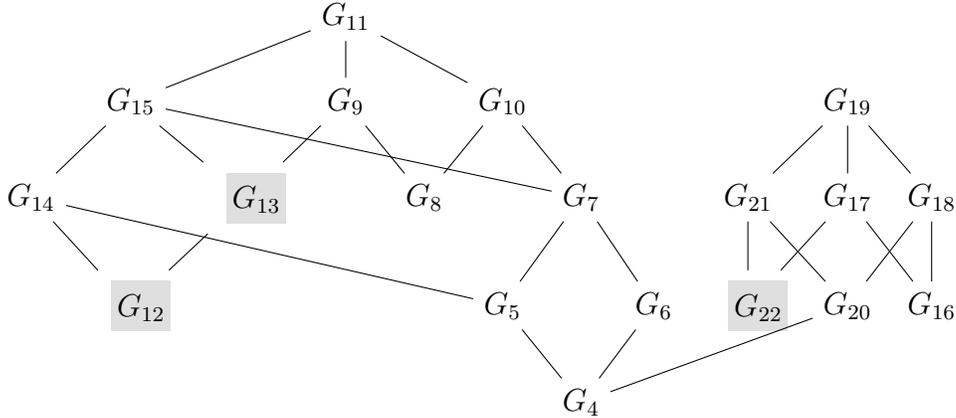
\begin{figure}[!h]
\begin{center}
\caption{The inclusions between the primitive complex reflection groups
$G_4,\ldots,G_{22}$. Those which turn out to have quaternionic systems of 
imprimitivity are shaded. }
\label{GjInclusions}
\begin{tikzpicture}
    \matrix (A) [matrix of nodes, row sep=0.5cm, column sep = 0.2 cm]
    { &&& $G_{11}$ &&&& & &&\\
      & $G_{15}$ && $G_{9}$ && $G_{10}$ &&& & $G_{19}$ &  \\
  $G_{14}$ && \greybox{$G_{13}$} && $G_{8}$ && $G_7$ && $G_{21}$ & $G_{17}$ & $G_{18}$ \\
  & \greybox{$G_{12}$} && && $G_5$&& $G_6$ & \greybox{$G_{22}$} & $G_{20}$ & $G_{16}$  \\
	& && && & $G_4$ & &&& \\
    };
    \draw (A-1-4)--(A-2-2); \draw (A-1-4)--(A-2-4); \draw (A-1-4)--(A-2-6);
    \draw (A-2-2)--(A-3-1); \draw (A-2-2)--(A-3-3); \draw (A-2-4)--(A-3-3);
    \draw (A-2-4)--(A-3-5); \draw (A-2-6)--(A-3-5); \draw (A-2-6)--(A-3-7);
    \draw (A-2-2)--(A-3-7); \draw (A-3-1)--(A-4-2); \draw (A-3-3)--(A-4-2);
    \draw (A-3-1)--(A-4-6); \draw (A-3-7)--(A-4-6); \draw (A-3-7)--(A-4-8);
    \draw (A-4-6)--(A-5-7); \draw (A-4-8)--(A-5-7); \draw (A-2-10)--(A-3-9); 
    \draw (A-2-10)--(A-3-10); \draw (A-2-10)--(A-3-11); \draw (A-3-9)--(A-4-9);
    \draw (A-3-9)--(A-4-10); \draw (A-3-10)--(A-4-9); \draw (A-3-10)--(A-4-11);
    \draw (A-3-11)--(A-4-10); \draw (A-3-11)--(A-4-11); \draw (A-4-10)--(A-5-7);
\end{tikzpicture}
\end{center}
\vskip-0.6truecm
\end{figure}

\begin{lemma}
\label{G4..G22sys}
If a primitive complex reflection group $G_{4},G_5,\ldots,G_{22}$ has a quaternionic
	system of imprimitivity, then it is given by $(1,q)$, where
$$ q=zj, \quad |z|=1, \quad z\in\CC. $$
\end{lemma}

\begin{proof} These groups have a common (imprimitive) subgroup
$$ G_4\cap G_5\cap\cdots\cap G_{22} =\inpro{\pmat{i&0\cr0&-i},\pmat{0&i\cr i&0}}. $$
For the above generators (which are not reflections), the equations (\ref{abcdqeqns}) 
of Lemma \ref{abcdq-lemma} for $(1,q)$, $q=a+bi+cj+dk$, 
to give a system of imprimitivity are
$$ i -\overline{q}iq=0, \qquad -qi -iq=0 \Iff
	i -\overline{q}iq=0, \qquad q=cj+dk=zj, \ (z\in\CC), $$
$$ iq +\overline{q}i =0, \qquad -qiq +i =0 \Iff q=bi, \qquad -qiq +i =0 . $$
Since the complex reflection groups are primitive, we cannot have $q=bi$,
so we must have $-qiq +i =0$ in the second set of equations.
Now
$$ -qiq +i =0, \quad i -\overline{q}iq=0 \Implies q=\overline{q} \Implies q\in\RR, $$
so we must have $q=zj$ in the first set of equations, and we calculate
$$ q=zj \Implies -qiq +i=(1-|z|^2)i=0 \Iff |z|=1, $$
giving the condition for $-qiq +i =0$ to hold in this case.
\end{proof}

We observe that $(zj)^2=-1$, so that $\overline{zj}=-zj$ 
and $(zj,1)$ is orthogonal to $(1,zj)$.
We will write the representation of the matrix $g$ with respect to 
the orthogonal basis 
$\{(1,zj),(zj,1)\}$ 
given by $q=zj$, $|z|=1$, $z\in\CC$, as
\begin{equation}
\label{zj-mat-rep}
[g]=U^{-1}gU, \qquad U={1\over\sqrt{2}}\pmat{1&zj\cr zj&1}, \quad
U^{-1}=U^*={1\over\sqrt{2}}\pmat{1&-zj\cr -zj&1}. 
\end{equation}

In view of Lemma \ref{G4..G22sys}, a primitive complex reflection 
group of rank two has a system of imprimitivity, necessarily given by $(1,zj)$, $|z|=1$,
$z\in\CC$, if and only if each of its generators is a monomial matrix 
in the representation of (\ref{zj-mat-rep}).
In this regard, we have

\begin{example}
\label{G12G13G22example}
The generators for $G_{12},G_{13},G_{22}$ of (\ref{G12G13G22defn}) in the representation (\ref{zj-mat-rep}) are
$$ [F] 
= {1\over\sqrt{2}}\pmat{0&1+zj\cr 1-zj&0}, \quad
[F^Z] 
= {1\over\sqrt{2}}\pmat{0&i+zj\cr -i-zj &0}, $$ 
$$ [F^{Z^2}] 
= {1\over\sqrt{2}}\pmat{0&1+i\cr 1-i &0}, \quad
[R^2]=\pmat{0&zj\cr-zj&0}, $$ 
$$ [A]=\pmat{0&zj\cr-zj&0}, \quad
[A^Z]=\pmat{0&i\cr-i&0}, \quad
[A^M]= \pmat{0&{-1-\gs i\over2}+{\tau zj\over2}\cr{-1+\gs i\over2}-{\tau zj\over2}  &0}. $$
and so these groups have systems of imprimitivity given by $q=zj$, $|z|=1$, $z\in\CC$.

The reflection system for the monomial representation of $G_{12}$ over $\HH$
given above, which is generated by 
$\{{1+zj\over\sqrt{2}},{i+zj\over\sqrt{2}},{1+i\over\sqrt{2}}\}$, has
$12$ elements, i.e., 
$$\cL=\hbox{$ \{
{1+zj\over\sqrt{2}},
{1-zj\over\sqrt{2}},
{-1+zj\over\sqrt{2}},
{-1-zj\over\sqrt{2}},
{i+zj\over\sqrt{2}},
{i-zj\over\sqrt{2}},
{-i+zj\over\sqrt{2}},
{-i-zj\over\sqrt{2}},
{1+i\over\sqrt{2}},
{1-i\over\sqrt{2}},
{-1+i\over\sqrt{2}},
{-1-i\over\sqrt{2}}
\} $}.  $$
If the set $\cL$ is multiplied by the inverse of any one of its elements,
to obtain an equivalent reflection system containing $1$, then the group 
$K$ generated by its elements has order $24$.
Thus $\cL$ is equivalent to $L_{12}^\cT$, which is the only quaternionic 
reflection system of that size for a group $K$ of order $24$ 
(either $\cT$ or $\cD_6$), and hence the 
primitive complex reflection group $G_{12}$ is conjugate to the
imprimitive quaternionic reflection group $G_\cT(L_{12}^\cT,1)$.
The reflection system for $G_{13}$ obtained by adding the extra generator $zj$ 
	from $[R^2]$ has 
six additional elements $\pm 1,\pm zj,\pm i$. It is equivalent to
$L_{18}^\cO$,
which is 
the unique quaternionic reflection system of size $18$,
and so the monomial representation for $G_{13}$ given above
is the imprimitive quaternionic reflection group $G_\cO(L_{18}^\cO,1)$.

Similarly, the reflection system for the monomial representation of $G_{22}$ 
is generated by $\{zj,i,{-1-\gs i\over2}+{\tau zj\over2}\}$ and 
has size $30$ (it contains $1$ and generates a group of
order $120$). It is therefore equivalent to $L_{30}^\cI$, which is
the only quaternionic reflection system of this size,
and $G_{22}$ is conjugate to the imprimitive quaternionic
reflection group $G_\cI(L_{30}^\cI,1)$.
\end{example}

We observe that the inclusion $G_{12}\subset G_{13}$ implies the inclusion
	$$ G_\cT(L_{12}^\cT,1) \subset G_\cO(L_{18}^\cO,1), $$
which is apparent from the generators for their reflection systems given
in Table \ref{TOIDngroups-table}.

\begin{example}
\label{collineationsystems}
The imprimitive complex reflection groups of rank two correspond
to the following three collineation groups 
$$ G_4,\ldots,G_7 \ \hbox{(tetrahedral)}, \quad
G_8,\ldots,G_{15} \ \hbox{(octahedral)}, \quad
G_{16},\ldots,G_{22} \ \hbox{(icosahedral)}. $$
Those with the same collineation group have the same complex systems
of imprimitivity (none in this case). However, this is not the case for 
	quaternionic systems of primitivity, as $G_{12}$, $G_{13}$ (octahedral)
	and $G_{22}$ (icosahedral) have infinitely many systems, whilst the
	other groups have none.
\end{example}

We can now determine all the quaternionic systems of imprimitivity 
of the primitive complex reflection groups, using the inclusions of
Figure \ref{GjInclusions} (see \cite{W26}) to expedite the proof.

\begin{theorem}
\label{complexprimitivesys}
The primitive complex reflection groups of rank two with quaternionic
systems of imprimitivity are
	$$ G_{12},\, G_{13}, \ \hbox{(octahedral type)} 
	\qquad G_{22} \ \hbox{(iscosahedral type)}. $$
Their systems of imprimitivity are given by $(1,zj)$, $|z|=1$, $z\in\CC$,
	and a corresponding change of basis matrix conjugates them
to the imprimitive quaternionic reflection (base) groups
$$ G_\cT(L_{12}^\cT,1), \quad
G_\cO(L_{18}^\cO,1), \quad
	G_\cI(L_{30}^\cI,1), $$
respectively. The remaining $16$ groups
$$ G_4,\, G_5,\, G_6,\, G_7, \quad
G_8,\, G_9,\, G_{10},\, G_{11},\, G_{14},\, G_{15}, \quad
G_{16},\, G_{17},\, G_{18},\, G_{19},\, G_{20},\, G_{21},  $$
have no quaternionic systems of imprimitivity.
\end{theorem}

\begin{proof}
The first part is given in Example \ref{G12G13G22example}.
It therefore suffices to show that the other groups have no systems of imprimitivity.
This can be proved (directly) by applying Lemma \ref{abcdq-lemma}
to each of the groups, 
or, equivalently, by finding the matrix representation (\ref{zj-mat-rep}) 
of their generators. However, in view of the inclusions of Figure \ref{GjInclusions}, 
it suffices to show that the groups
$$ G_4=\inpro{Z,Z^S}, \qquad G_8=\inpro{R,R^F}, \qquad
G_{16}=\inpro{M,M^A}, $$
have no systems of imprimitivity, and hence nor do the groups which contain them.

If the group $G=G_4,G_8,G_{16}$ had a quaternionic system of imprimitivity,
then the corresponding
imprimitive
quaternionic reflection group would have the same order and same number of reflections.
In particular, 
since the  imprimitive quaternionic reflection groups all have reflections of order two
corresponding to the reflection system $\cL$, the group $G$ must have
reflections of order two. Since all the reflections of $G_4$ and $G_{16}$
have orders $3$ and $5$, respectively, 
we only need consider $G_{8}=\inpro{R,R^F}$. 
Since $R$ is a monomial reflection of order $4$,
Lemma \ref{imprimitivitylemma} (with $h=-i$) 
	gives
	that group generated by $R$ has only the system of imprimitivity given by $(1,0)$.
Therefore, $G_8$, which contains $R$ 
and does not have a system of imprimitivity given by $(1,0)$, has no quaternionic
systems of imprimitivity, and we are done.

Alternatively, to show $G_4,G_8,G_{16}$ have no systems of imprimitivity directly 
from Lemma \ref{abcdq-lemma}, 
we can use the fact that $q=zj$ and $a,b,c,d\in\CC$ to simplify 
(\ref{abcdqeqns}) to
$$ (a+\overline{d})+(b-\overline{c})zj=0, \qquad
-(\overline{b}+c)+(\overline{a}-d)zj=0, $$
i.e., $ a=-\overline{d}, \ b=\overline{c}$, or
 $a=\overline{d}, \ b=-\overline{c}$.
These are easily seen to not hold for the generators $g=Z,R,M$.
\end{proof}

\begin{table}[h]
\caption{The systems of imprimitivity for the complex reflection groups 
$G_4,\ldots,G_{22}$ (primitive) and $G(n,p,2)$ (imprimitive) 
of (\ref{Gnp2defn}), as
classified by Shephard and Todd. The real reflection group $G(4,4,2)\cong_\CC G(2,1,2)=D_4$ is
included for 
	comparison.}
\label{ComplexSystems-table}
\begin{center}
\vskip-0.55truecm
\begin{tabular}{ |  >{$}l<{$} | >{$}c<{$} | >{$}l<{$} | l | }
\hline
&&&\\[-0.3cm]
	G & \hbox{complex} & \hbox{quaternionic} & comments \\[0.1cm]
\hline
&&&\\[-0.3cm]
G_{12} & & (1,zj),\ z\in\CC,\ |z|=1 & $\cong_\HH Q_\cT(L_{12}^\cT,1)$  \\
G_{13} & & (1,zj),\ z\in\CC,\ |z|=1 & $\cong_\HH G_\cO(L_{18}^\cO,1)$  \\
G_{22} & & (1,zj),\ z\in\CC,\ |z|=1 & $\cong_\HH G_\cI(L_{30}^\cI,1)$  \\
G_{n},\ n\ne 12,13,22 & & & no systems  \\
&&& \\[-0.3cm]
\greybox{G(4,4,2)} & (1,0),(1,1),(1,i) & (1,zj),\ z\in\CC,\ |z|\le1 & real reflection group \\
G(4,2,2) & (1,0),(1,1),(1,i) & (1,zj),\ z\in\CC,\ |z|=1 & $\cong_\HH G(2,1,2,1)$ \\
G(n,n,2),\ n\ne4 & (1,0) & (1,zj),\ z\in\CC,\ |z|\le1 & \\
G(n,{n\over 2},2),\ n\ne4\, \hbox{($n$ even)} & (1,0) & (1,zj),\ z\in\CC,\ |z|=1 & $\cong_\HH G({n\over2},1,{n\over2},1)$ \\ 
G(n,p,2),\ n\ne4, p\ne n,{n\over2}  & (1,0) & & no additional systems \\ [3pt] 
\hline
\end{tabular}
\end{center}
\end{table}

From Table \ref{ComplexSystems-table}, we observe that the only 
complex reflection groups of rank two with additional complex systems of
imprimitivity are $G(4,4,2)\cong_\CC G(2,1,2)$ (real reflection group) and 
$G(4,2,2)$ (see Theorem 2.16 of \cite{LT09}).

\section{The systems of imprimitivity of the quaternionic reflection groups}

We now consider the systems of imprimitivity for the quaternionic
reflection groups.

The primitive quaternionic reflection groups of rank two consist of
six with primitive complexifications (see \cite{C80}, \cite{W24}, \cite{BW25}),
and an infinite family of those with imprimitive complexifications (see \cite{C80}, \cite{T25}).
All of these, by definition of being primitive, have no quaternionic 
systems of imprimitivity.

The imprimitive quaternionic reflection groups are given the Table \ref{TOIDngroups-table},
and we begin with those with group $K=\cD_n$. 
We define the {\bf dicyclic groups} (binary dihedral groups) 
$$ \cD_n := \inpro{\go,j}, \qquad \go:=\zeta_{2n} = e^{\pi i\over n}, \qquad n\ge 2,
$$ 
where $\go$ is a primitive $2n$-th root of unity.
This group has $4n$ has elements, of two types
\begin{equation}
\label{Dntwotypes}
\go^m, \qquad \go^\ell j=j\go^{-\ell}, \qquad 1\le m,\ell \le 2n.
\end{equation}
The group $\cD_2$ is the quaternion group 
$Q_8=\{1,-1,i,-i,j,-j,k,-k\}$, 
which has a slightly special 
structure, since $\go=i$, so that $i$, $j$, $k$ play the same role,
and this will lead to additional systems of imprimitivity in this case
(see Example \ref{KMUBexample}).

The reflection systems for $\cD_n$ are given in \cite{W25} in terms
of generators as
$$ L_{(a,b)}^{(n)} := L(\{1,\go^a,j,\go^b j\}),
\qquad (a,b)\in \gO_n,  $$
where
\begin{equation}
\label{Andefn}
\gO_n  := \{(a,b):1\le a\le b\le n,\ a\divides n,\  b\divides n,\ \gcd(a,b)=1\}.
\end{equation}
Each of these has a different number of elements, which is given by
\begin{equation}
\label{Labnsize}
|L_{(a,b)}^{(n)}|={2n\over a}+{2n\over b}.
\end{equation}
The corresponding imprimitive quaternionic reflection groups
with $H$ cyclic 
are
\begin{equation}
\label{Gnabrdefn}
G(n,a,b,r):= G_{\cD_n}(L_{(a,b)}^{(n)},C_r)
=\cG(\{1,\go^a,j,\go^b j\},\{\go^{2n\over r}\}), \qquad [n,a,b,r]\in\gL_n,
\end{equation}
where
\begin{equation}
\label{gLndefn}
\gL_n= \bigcup_{(a,b)\in\gO_n}\{ \hbox{$[n,a,b,{n\over ab}]$}\} \cup 
\bigcup_{(a,b)\in\gO_n\atop ab\, {\rm is}\, {\rm odd}} \{\hbox{$[n,a,b,{2n\over ab}]$}\}.
\end{equation}
The reflection group $G(n,1,n,1)=\cG(\{1,\go,j,-j\},\{\})=\cG(\{1,\go,j\},\{\})$ 
will be seen to be conjugate to the imprimitive complex reflection group 
$G(2n,n,2)=\cG(\{1,\go\},\{-1\})$, see Example \ref{G(n,1,n,1)iscomplex}, 
and so we exclude its index $[n,1,n,1]$. 
For $L_{(1,1)}^{(n)}=\cD_n$, there are also reflection groups with $H$ not cyclic, 
namely,
\begin{align}
\label{DnHDn}
G(\cD_n,\cD_n,\cD_n) &=\cG(\{1,\go,j,\go j\},\{\go,j\}), \cr
G(\cD_n,\cD_n,\cD_{n/2}) &=\cG(\{1,\go,j,\go j\},\{j\}),
\qquad \hbox{($n$ even, $n\ge4$)}.
\end{align}
Therefore the imprimitive quaternionic reflection groups for $K=\cD_n$ are
those of (\ref{DnHDn}) and 
\begin{equation}
\label{Gnabrquaternionic}
G(n,a,b,r), \qquad [n,a,b,r]\in\gL_n^*:=\gL_n\setminus\{[n,1,n,1]\},
\end{equation}
where
$$ |\gL_n^*|={\tau(2n^2)\over2}, $$
with $\tau(2n^2)$ the number of divisors of $2n^2$.

We now give the quaternionic systems of imprimitivity for the
groups $G(n,a,b,r)$, including $G(n,1,n,1)$, for the
purpose of comparison.


\begin{theorem}
\label{G(n,a,b,r)sys}
Let $n\ge2$. The imprimitive 
reflection group $G(n,a,b,r)$, $[n,a,b,r]\in\gL_n$,
has an additional 
system of imprimitivity given by 
	$(1,q)\in\HH^2$ 
in precisely the 
cases:
\begin{enumerate}[(a)]
\item $(1,j)$, $(1,k)$ for the indices $[n,1,n,1]$, $[n,1,n,2]$ ($n$ odd),
  $[n,1,{n\over2},2]$ ($n$ even), and $[n,2,{n\over2},1]$ ($n$ even, ${n\over2}$ odd).
\item $(1,1)$, $(1,i)$ for the indices 
$[ 2, 1, 2, 1 ]$, 
$[ 2, 1, 1, 2 ]$. 
\end{enumerate}
and the infinite family
\begin{enumerate}[(a)]
\setcounter{enumi}{2}
\item $(1,\ga k)$, $-1<\ga< 1$ ($\ga\ne0$), for the indices $[n,1,n,1]$.
\end{enumerate}
We have $G(n,1,n,1)\cong_\HH G(2n,n,2)$, and so the
group  $G(n,1,n,1)$ above is counted as an imprimitive complex reflection group
	(not a quaternionic one).

The other imprimitive quaternionic reflection groups $G(\cD_n,\cD_n,\cD_n)$
and  $G(\cD_n,\cD_n,\cD_{n/2})$ have no
additional systems of imprimitivity.
\end{theorem}

\begin{proof} We apply Lemma \ref{imprimitivitylemma}, with
$$ \cL=\{1,\go^a,j,\go^b j\}, \qquad
	\cH=\{\go^{2n\over r}\}. $$
The two possibilities of (ii) are
\begin{align}
\label{req2cdnonh}
\go^{2n\over r}=-|q|^2 & \Iff \go^{2n\over r}=-1, \quad |q|=1 \cr
	&\Iff r=2, \quad |q|=1, \cr
	\go^{2n\over r}=1 & \Iff r=1,
\end{align}
which gives the necessary condition $r=1,2$. We observe that for $r=1$, there is
as yet no restriction on $|q|$.
Taking $\gb=1,\go^a,j,\go^b j$ in (i), respectively, 
gives the following two conditions, one of which must hold for there
to be a system of imprimitivity,  
\begin{align*}
\Re(q)=0, & \qquad q=\pm1, \cr
\Re(\go^a q)=0, & \qquad q=\pm\go^{-a}, \cr
\Re(jq)=0, & \qquad q=\pm j, \cr
\Re(\go^b j q)=0, & \qquad q=\pm\go^{b} j.
\end{align*}

First suppose that $q=\pm1$, 
then $\Re(j q)=\pm\Re(j)=0$, $\Re(\go^b jq)=\pm\Re(\go^b j)=0$ hold,
and so to obtain a system of imprimitivity one of the
following must hold
	$$ \Re(\go^a)=0, \quad \go^{2a}=1 \Iff \go^{2a}=(\go^a)^2=(\pm i)^2=-1, \quad \go^{2a}=1, $$
i.e., $2a=n$ or $2a=2n$. We cannot have $a=n$, which would imply 
$\gcd(a,b)=n>1$, and so we require $a={n\over2}$, which implies
$$ a={n\over2}\le b\le {n\over n/2}=2. $$
We cannot have $a=2$, which would imply $\gcd(a,b)=2$, and so we require
	$$ a={n\over2}=1 \Implies n=2. $$
There are two indices in $\gL_2$ of the form $[2,1,b,r]$, $r=1,2$, namely 
	$[2,1,2,1]$, $[2,1,1,2]$, and so we obtain the indices of (b).

Henceforth, we can suppose that $\Re(q)=0$. 
Suppose that $q=\pm\go^a$, then we have
$$ \pm\go^a=\pm i \Iff \go^{2a}=-1, $$
which is satisfied by the indices $[2,1,2,1]$, $[2,1,1,2]$, as before,
and (b) is proved, since this $q$ satisfies 
$\Re(j q)=0$ and $\Re(\go^b j q)=0$.

Now we may suppose that $\Re(q)=0$, $\Re(\go^a q)=0$. Consider the case
$q=\pm j$, which satisfies these conditions. To satisfy the last, 
we must have one of
	$$ \Re(\go^b)=0, \quad \go^b=\pm 1 \Iff
	b={n\over 2}, \quad b=n. $$
For $b=n$, we have $a=1$, and $r=1$ (base group) and $r=2$ for $n$ odd (higher
	order group).
For $b={n\over2}$ ($n$ even), we can have $a=1$, which gives $r=2$ (base group), 
and $a=2$ when ${n\over2}$ is odd, which gives $r=1$ and there is no higher
order group.

Finally, we may suppose $\Re(q)=0$, $\Re(\go^a q)=0$, $\Re(jq)=0$.
If $\Re(\go^bjq)=0$, then 
$$ \Re(\go^bjq)=0 \Iff \go^bjq=\pm i|q|
\Iff q=\pm |q| \go^{b-{n\over2}} j. $$
Thus we have the two cases
$$ q=\pm |q| \go^{b-{n\over2}} j, \qquad q=\pm \go^{b} j, $$
which both satisfy $\Re(q)=0$, $\Re(\go^a q)=0$, and for a system of 
imprimitivity to exist we must have
$\Re(jq)=0$, i.e., 
\begin{align*}
\Re(\go^{{n\over2}-b})=0, \quad \Re(\go^{-b})=0
&\Iff \go^{{n\over2}-b}=\pm i, \quad \go^{-b}=\pm i \cr
&\Iff \go^{n-2b}=-\go^{-2b}=-1, \quad \go^{-2b}=-1 \cr
&\Iff -2b=2n, \quad n=-2b \cr
	&\Iff b=n, \quad b={n\over2},
\end{align*}
for which have previously determined the possible indices. We observe that in these
cases $q$ is given by
$$ q=\pm |q| \go^{n-{n\over2}} j=\pm|q|ij=\pm|q|k, \qquad q=\pm \go^{n\over2} j=\pm ij=\pm k, $$
and so, together with the previous case, we obtain (a). The only case where $|q|$ is 
not restricted to the value $1$ by (\ref{req2cdnonh}) or the choice $q=\pm\go^b j$ is for 
the index $[n,1,n,1]$, which gives (c).

The conjugacy $G(n,1,n,1)\cong_\HH G(2n,n,2)$ is considered in Example \ref{G(n,1,n,1)iscomplex} (to follow).
The groups $G(\cD_n,\cD_n,\cD_n)$ and $G(\cD_n,\cD_n,\cD_{n/2})$ contain contain
the reflection given by $h=\go^2$, which does not satisfy the 
condition (ii) of Lemma  \ref{imprimitivitylemma}, and so these
groups have no systems of imprimitivity.
\end{proof}

Theorem \ref{G(n,a,b,r)sys} 
corresponds to the Theorem 6.4 of \cite{T25}, as we now explain.

\begin{example}
In 
\cite{T25}, 
the groups $G(n,a,b,r)$ are described via the ``standard copy''
\begin{equation}
\label{TGnabr} 
G(\cD_n,\cC_r,\psi_c)
=\inpro{\pmat{0&1\cr1&0},\pmat{0&j\cr-j&0},\pmat{\go^{2n\over r}&0\cr0&1},
\pmat{\go&0\cr0&\go^c} },
\end{equation}
indexed by $(n,r,c)$,
where $c=1$, or $1<c\le {n\over r}$ ($r\divides 2n$) and
$$ \gcd(c,{2n\over r})=\gcd(\nu,\kappa)=1, \qquad
\nu:={{2n\over r}\over\gcd({2n\over r},c-1)}, \quad
\kappa:={{2n\over r}\over\gcd({2n\over r},c+1)}. $$
We observe that $c=1$ also satisfies the above condition.
The fourth generator in (\ref{TGnabr}) is not a reflection. 
From Lemma 4.8 and Lemma 4.10 of \cite{T25}, it appears that the reflection system
for $G(\cD_n,\cC_r,\psi_c)$ contains $\go^\kappa$ and $\go^\nu j$
(in addition to $1,j$), so that
	$$ G(\cD_n,\cC_r,\psi_c)=\cG(\{1,\go^\kappa,j,\go^\nu j\},\{\go^{2n\over r}\}), $$
and we have the following correspondence between the respective indices
$$ [n,a,b,r]\in\gL_n \iff (n,r,c), \qquad
\{a,b\}=\{\nu,\kappa\}=\{ {{2n\over r}\over\gcd({2n\over r},c-1)},
{{2n\over r}\over\gcd({2n\over r},c+1)} \}. $$
Moreover, the groups $G(n,a,b,r)$ and $G(\cD_n,\cC_r,\psi_c)$ are equal
	when $\kappa\le\nu$, and are isomorphic (but not equal) when $\kappa>\nu$.
In the latter case, 
$$ G(n,a,b,r)
=\inpro{\pmat{0&1\cr1&0},\pmat{0&j\cr-j&0},\pmat{\go^{2n\over r}&0\cr0&1},
	\pmat{\go&0\cr0&\go^{-c}} }, \qquad \kappa>\nu, $$
and, by (\ref{RefSys-equiv}), we have
$$ G(n,a,b,r)= U G(\cD_n,\cC_r,\psi_c) U^{-1}, \qquad U=\pmat{1&0\cr0& j}. $$
In this way, the Theorem 6.4 of \cite{T25} is seen to be equivalent to
Theorem \ref{G(n,a,b,r)sys}. 
\end{example}

\begin{example} 
\label{G(n,1,n,1)iscomplex}
We consider the group
$$ G(n,1,n,1)=\cG(\{1,\go,j,-j\})=\cG(\{1,\go,j\})
=\inpro{\pmat{0&1\cr1&0},\pmat{0&\go\cr\overline{\go}&0},\pmat{0&j\cr-j&0}}. $$
For the system of imprimitivity given by $(1,j)$, we take
the change of basis matrix 
$$U={1\over\sqrt{2}}\pmat{1&j\cr j&1}. $$
The conjugation (change of basis) $g\mapsto U^{-1}gU$ applied to 
above generators gives
$$ \pmat{0&1\cr1&0}\mapsto \pmat{0&1\cr1&0}, \qquad
 \pmat{0&\go\cr\overline{\go}&0}\mapsto \pmat{0&\go\cr\overline{\go}&0}, \qquad
 \pmat{0&j\cr-j&0}\mapsto \pmat{-1&0\cr0&1},  $$
i.e., the generators of the imprimitive  complex reflection group
$$ G(2n,n,2)=\cG(\{1,\go\},\{-1\}), $$
which is therefore conjugate to $G(n,1,n,1)$.

Further, the other groups of Theorem \ref{G(n,a,b,r)sys} (a) with this system of imprimitivity are
\begin{align*}
G(n,1,n,2) &=\cG(\{1,\go,j\},\{-1\}), \cr
G(n,1,{n\over 2},2) &=\cG(\{1,\go,j,ij\},\{-1\}), \cr
	G(n,2,{n\over 2},1) &=\cG(\{1,\go^2,j,ij\},\{\}), 
\end{align*}
with the generators for $h=-1\in\cH$ and $\gb=ij=k\in\cL$ conjugating as follows
$$ \pmat{-1&0\cr0&1} \mapsto \pmat{0&-j\cr j&0}, \qquad
\pmat{0&k\cr-k&0}\mapsto\pmat{0&k\cr-k&0}. $$
Therefore none of these groups conjugate to a complex reflection group.
\end{example}

The reflection systems for the imprimitive {\em quaternionic} reflection groups
given by Theorem \ref{G(n,a,b,r)sys} are summarised in the first section of Table \ref{QuaternionicSystems-table}.

We now consider the groups $G_K(L,H)$ for $K$ the {\bf binary tetrahedral}, {\bf
octahedral} and {\bf icosahedral groups} which given by (see Table \ref{TOIDngroups-table})
\begin{align*}
\cT &:=\inpro{i,j,{1+i+j+k\over2}}, \cr
\cO &:=\inpro{{1+i\over\sqrt{2}},{1+j\over\sqrt{2}},{1+i+j+k\over2}},  \cr
\cI &:=\inpro{{1+i\over\sqrt{2}},{1+i+j+k\over2},{\tau+\gs i-j\over2}}, \quad
	\tau={1+\sqrt{5}\over2},\ \gs=1-\tau.
\end{align*}

\vskip-0.7truecm
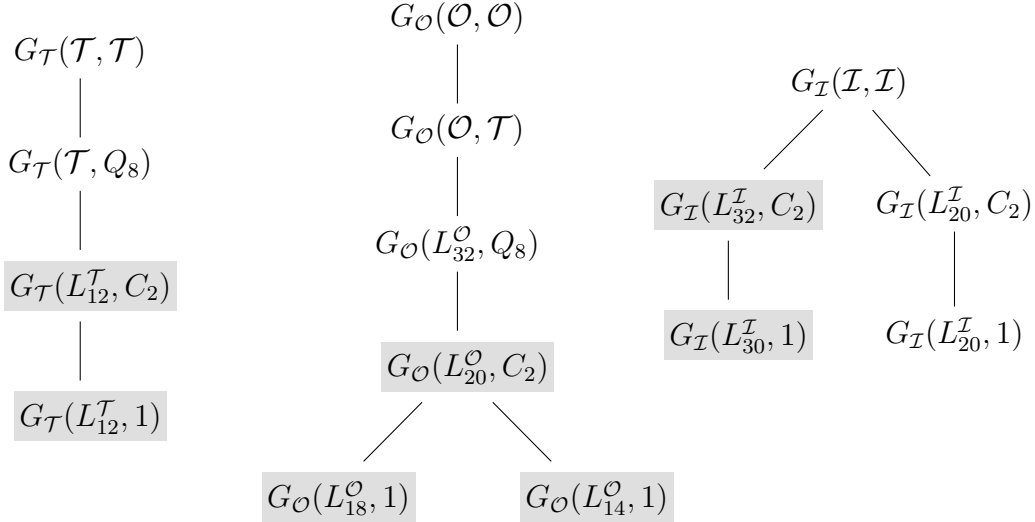
\begin{figure}[!h]
\caption{The inclusions of the reflection groups for $\cT,\cO,\cI$ given in Table \ref{TOIDngroups-table},
	with those having additional systems of imprimitivity shaded in grey. }
\label{TOIinclusionsfigure}
\begin{center}
	\begin{tabular}[t]{ p{3.0truecm} p{4.8truecm} p{5.8truecm} }
\begin{tikzpicture}
\matrix (A) [matrix of nodes, row sep=0.8cm, column sep = 0.0 cm]
    {
	    $G_\cT(\cT,\cT)$  \\
	    $G_\cT(\cT,Q_8)$  \\
	    \greybox{$G_\cT(L_{12}^\cT,C_2)$}  \\
	    \greybox{$G_\cT(L_{12}^\cT,1)$}  \\
	    $ $ \\
    };
    \draw (A-1-1)--(A-2-1);
    \draw (A-2-1)--(A-3-1);
    \draw (A-3-1)--(A-4-1);
\end{tikzpicture}
 & 
\begin{tikzpicture}
    \matrix (A) [matrix of nodes, row sep=0.8cm, column sep = -1.0 cm]
    {
	    & $G_\cO(\cO,\cO)$ &  \\
	    & $G_\cO(\cO,\cT)$ &  \\
	    & $G_\cO(L_{32}^\cO,Q_8)$ &  \\
	    & \greybox{$G_\cO(L_{20}^\cO,C_2)$} &  \\
	    \greybox{$G_\cO(L_{18}^\cO,1)$} && \greybox{$G_\cO(L_{14}^\cO,1)$} \\
    };
    \draw (A-1-2)--(A-2-2);
    \draw (A-2-2)--(A-3-2);
    \draw (A-3-2)--(A-4-2);
    \draw (A-4-2)--(A-5-1);
    \draw (A-4-2)--(A-5-3);
\end{tikzpicture}
	  & 
		\begin{tikzpicture}
    \matrix (A) [matrix of nodes, row sep=0.8cm, column sep = -0.7 cm]
    {
	    & $G_\cI(\cI,\cI)$ &  \\
			\greybox{$G_\cI(L_{32}^\cI,C_2)$} && $G_\cI(L_{20}^\cI,C_2) $  \\
			\greybox{$G_\cI(L_{30}^\cI,1)$} && $G_\cI(L_{20}^\cI,1) $  \\
	    $ $ \\
	    $ $ \\
    };
    \draw (A-1-2)--(A-2-1);
    \draw (A-1-2)--(A-2-3);
    \draw (A-2-1)--(A-3-1);
    \draw (A-2-3)--(A-3-3);
\end{tikzpicture}
\end{tabular}
\end{center}
\vskip-1.2cm
\end{figure}

\vskip-1.5truecm
\begin{theorem}
\label{TOIsys}
Let $\cT,\cO,\cI$ be the binary tetrahedral, octahedral and icosahedral groups.
Then the associated quaternionic reflection groups, 
as listed in Table \ref{TOIDngroups-table},
have additional systems of imprimitivity given by $(1,q)\in\HH^2$ in
	the following cases:
\begin{enumerate}[(a)]
\item $(1,{j-k\over\sqrt{2}})$ for $G_\cT(L_{12}^\cT,C_2)$, 
	$G_\cO(L_{14}^\cO,1)$, $G_\cO(L_{20}^\cO,C_2)$.
\item $(1,{j-\tau i-\gs k\over2})$ for $G_\cI(L_{32}^\cI,C_2)$.
\end{enumerate}
and the infinite families
\begin{enumerate}[(a)]
\setcounter{enumi}{2}
\item $(1,\ga {j-k\over\sqrt{2}})$, $-1<\ga\le1$ ($\ga\ne0$) for $G_\cT(L_{12}^\cT,1)$,
	$G_\cO(L_{18}^\cO,1)$.
\item $(1,\ga {j-\tau i-\gs k\over2})$, $-1<\ga\le1$ ($\ga\ne0$) for $G_\cI(L_{30}^\cI,1)$.
\end{enumerate}
For the groups above giving infinite families, we have the conjugacies
\begin{equation}
\label{primitiveimprimitiveisos}
G_{\cT}(L_{12}^\cT,1)\cong_\HH G_{12}, \qquad
G_{\cO}(L_{18}^\cO,1)\cong_\HH G_{13}, \qquad
G_{\cI}(L_{30}^\cI,1)\cong_\HH G_{22}.
\end{equation}
\end{theorem}

\begin{proof} We apply Lemma \ref{imprimitivitylemma}, with $\cL$ and $\cH$
as given in Table \ref{TOIDngroups-table}.
Since 
$$ G(K_1,L_1,H_1) \subset G(K_2,L_2,H_2), \qquad
	\hbox{for}\quad	K_1\subset K_2,\ L_1\subset L_2,\ H_1\subset H_2, $$
it follows that a system of imprimitivity for $G(K_2,L_2,H_2)$ 
is a system of imprimitivity for $G(K_1,L_1,H_1)$. 
Thus, it suffices to start at the bottom of the lattice of inclusions for the
$K=\cT,\cO,\cI$ groups given in Figure \ref{TOIinclusionsfigure}, 
working upwards until a group with 
no additional systems of imprimitivity is identified, at which point
all the additional systems of imprimitivity have been found.

{\em The case $K=\cT$.} The lattice is linear, with bottom element 
$G_\cT(L_{12}^\cT,1)$ of order $48$ given by
$$ \cL=\{1,i,{1+i+j+k\over2}\}, \qquad \cH=\{\}. $$
The pairs of conditions on $q$ given by (i) of Lemma \ref{imprimitivitylemma} are
\begin{align*}
\gb=1: & \qquad \Re(q)=0, \quad q=\pm 1, \cr
\gb=i: & \qquad \Re(iq)=0, \quad q=\pm i, \cr
\gb={1+i+j+k\over2}: & \qquad \Re({1+i+j+k\over2}q)=0, \quad q=\pm {1-i-j-k\over2},
\end{align*}
and one of each pair must hold. This is not possible for any of the choices
for $q$ (which are mutually exclusive), e.g., $q=\pm1$, gives
$$ \Re({1+i+j+k\over2}q)=\pm{1\over2}\ne0. $$
Thus $q=a+bi+cj+dk$ must satisfy
\begin{equation}
\label{threeqcdns}
\Re(q)=0, \quad \Re(iq)=0, \quad \Re({1+i+j+k\over2}q)=0,
\end{equation}
i.e., by (\ref{abcdeqn}), 
$$ a=0, \quad -b=0, \quad a-b-c-d=0 \Implies q=\ga{j-k\over\sqrt{2}}, \quad\ga\in\RR. $$
The group $G_\cT(L_{12}^\cT,C_2)$ above $G_\cT(L_{12}^\cT,1)$ is obtained by adding 
$-1$ to $\cH$, which satisfies (ii), i.e., $h=-1=-|q|^2$ for the choice $\ga=1$,
	giving $q={j-k\over\sqrt{2}}$.
The group $G_\cT(L_{20}^\cT,Q_8)$ above $G_\cT(L_{12}^\cT,C_2)$ 
is obtained by adding $\gb=j$ to $\cL$.
This $\gb$ does not satisfy (i), i.e., 
\begin{equation}
\label{jexclude}
\Re(\gb q)= \Re(j{j-k\over\sqrt{2}})=-{1\over\sqrt{2}}\ne0, \qquad
q={j-k\over\sqrt{2}} \ne \pm j = \pm\overline{\gb}.
\end{equation}
Thus $G_\cT(L_{20}^\cT,Q_8)$ and $G_\cT(L_{20}^\cT,\cT)$ (the group above) have
no additional systems of imprimitivity.

{\em The case $K=\cO$}. The lattice has two minimal elements:
$G_\cO(L_{14}^\cO,1)$ and $G_\cO(L_{18}^\cO,1)$. 
The group $G_\cO(L_{14}^\cO,1)$ has
$$ \cL=\{1,i,{1+i+j+k\over2},{j-k\over\sqrt{2}}\}, \qquad \cH=\{\}, $$
which is the same as for $G_\cT(L_{12}^\cT,1)$, except for the addition 
of $\gb=q={j-k\over\sqrt{2}}$ to $\cL$, which, by construction, satisfies (\ref{threeqcdns}),
and the one of the conditions
$$ \Re(\gb q)=0, \quad q=\pm \overline{\gb}, $$
namely the last, so that there is a system of imprimitivity given by
$q={j-k\over\sqrt{2}}$. 

The other minimal element $G_\cO(L_{18}^\cO,1)$ has
$$ \cL=\{1,{1+i\over\sqrt{2}},{1+i+j+k\over2}\}, \qquad \cH=\{\}. $$
	The conditions which must be satisfied are (one of each of)
$$ \Re(q)=0, \quad q=\pm 1, $$
$$ \Re({1+i\over\sqrt{2}}q)=0, \quad q=\pm {1-i\over\sqrt{2}}, $$
$$ \Re({1+i+j+k\over2} q)=0, \quad q=\pm {1-i-j-k\over2}, $$
and these cannot be satisfied for any of the choices of $q$, thus
we must have
$$ \Re(q)=0, \quad \Re({1+i\over\sqrt{2}}q)=0, \quad\Re({1+i+j+k\over2} q)=0
\Implies q=\ga{j-k\over\sqrt{2}}.$$
We now consider the group $G_\cO(L_{20}^\cO,C_2)$ which contains the two 
groups considered, and so has a system of imprimitivity 
given by $q={j-k\over\sqrt{2}}$ or none. Its $\cL$ is obtained from that
for $G_\cO(L_{18}^\cO,1)$ by adding ${j-k\over\sqrt{2}}$, and 
so it has a system of imprimitivity given by $q={j-k\over\sqrt{2}}$.
The group $G_\cO(L_{32}^\cO,Q_8)$ above $G_\cO(L_{20}^\cO,C_2)$ is
obtained by adding $\gb=j$ to $\cL$, but this does not satisfy
the condition (i), as per (\ref{jexclude}),
and so there a no further systems of imprimitivity for the $\cO$ groups.

{\em The case $K=\cI$}. There are two minimal groups $G_\cI(L_{20}^\cI,1)$
and $G_\cI(L_{30}^\cI,1)$. The $\cL$ for $G_\cI(L_{20}^\cI,1)$ is obtained 
from that for $G_\cT(L_{12}^\cT,1)$ by adding $\gb={1+\gs j+\tau k\over2}$,
so we must have $q=\ga{j-k\over\sqrt{2}}$, but this choice does not satisfy 
(i), i.e., 
$$ \Re(\gb q)=\Re({1+\gs j+\tau k\over2}\ga {j-k\over\sqrt{2}})
=\ga{\tau-\gs\over2}\ne0, \quad q
\ne \pm {1+\gs j+\tau k\over2}, $$
and so $G_\cI(L_{20}^\cI,1)$ has no additional systems of imprimitivity.
Since  $G_\cI(L_{20}^\cI,C_2)$ and $G_\cI(\cI,\cI)$ are above $G_\cI(L_{20}^\cI,1)$
they have no additional systems of imprimitivity.

For $G_\cI(L_{30}^\cI,1)$,
$$ \cL=\{1,{1+i+j+k\over2},{\tau+\gs i-j\over2}\}, $$
none of the choices for $q$ in (i) works, 
and so we seek a $q=a+bi+cj+dk$ satisfying
$$ \Re(q)=0, \quad \Re({1+i+j+k\over2} q)=0, \quad \Re({\tau+\gs i-j\over2})=0, $$
i.e.,
$$ a=0, \quad a-b-c-d=0, \quad \tau a-\gs  b+c=0 
\Implies q=\ga {j-\tau i-\gs k\over2}. $$
The group $G_\cI(L_{32}^\cI,C_2)$ above $G_\cI(L_{30}^\cI,1)$ has its
$\cL$ obtained by adding $\gb={j-\tau i-\gs k\over2}$ to that for $G_\cI(L_{30}^\cI,1)$
(both have $\cH=\{\}$),
so that it has a system of imprimitivity given by $q={j-\tau i-\gs k\over2}$,
and we are finished.
\end{proof}

We observe that the isomorphism $G_\cO(L_{14},1)\to G_\cT(L_{12}^\cT,C_2)$ of
\cite{W25} (Example 4.4) is given by a system of imprimitivity (this is 
Theorem 7.1 of \cite{T25}).

\begin{example}
\label{sysimpconjex}
We have (see Table \ref{TOIDngroups-table}) that
$$ L_{14}^\cO = L_{12}^\cT \cup \{{j-k\over\sqrt{2}},{k-j\over\sqrt{2}}\}, $$
	where $L_{12}^\cT$ is generated by $\{1,i,{1+i+j+k\over2}\}$.
For the second reflection system for $G_\cO(L_{14}^\cO,1)$ given by $(1,{j-k\over\sqrt{2}})$,
take the change of basis matrix
$$ U={1\over\sqrt{2}}\pmat{1&{j-k\over\sqrt{2}}\cr{j-k\over\sqrt{2}}&1}. $$
Then $U^{-1}G_\cO(L_{14},1) U= G_\cT(L_{12}^\cT,C_2)$, 
i.e., $G_\cO(L_{14},1)\cong_\HH G_\cT(L_{12}^\cT,C_2)$, since
$$ U^{-1}\pmat{0&{j-k\over\sqrt{2}}\cr {k-j\over\sqrt{2}}&0} U
	=\pmat{-1&0\cr0&1}, \qquad U^{-1}\pmat{0&b\cr b^{-1}&0} U=\pmat{0&b\cr b^{-1}&0},
	\quad b\in L_{12}^\cT. $$
\end{example}

\begin{table}[h!]
	\caption{The systems of imprimitivity for the (imprimitive) quaternionic
	reflection groups of rank two (see Theorems \ref{G(n,a,b,r)sys}
	and \ref{TOIsys}). Note that
	$G_\cT(L_{12}^\cT,C_2)\cong_\HH G_\cO(L_{14}^\cO,1)$.}
\label{QuaternionicSystems-table}
\begin{center}
\vskip-0.55truecm
\begin{tabular}{ |  >{$}l<{$} | >{$}l<{$} | l | }
\hline
&&\\[-0.3cm]
        G & \hbox{quaternionic} & \hbox{comments} \\[0.1cm]
\hline
&&\\[-0.3cm]
G(n,1,n,2)\ \hbox{($n$ odd)} & (1,0),\, (1,j),\, (1,k) & three systems \\
G(n,1,{n\over2},2)\ \hbox{($n$ even, $n\ne 2$)} & (1,0),\, (1,j),\, (1,k) & three systems \\
G(n,2,{n\over2},1)\ \hbox{($n$ even, ${n\over2}$ odd)} & (1,0),\, (1,j),\, (1,k) & three systems \\
G(2, 1, 1, 2) & (1,0),\, (1,1),\, (1,i),\, (1,j),\, (1,k) & five systems \cite{BW25} \\ 
	G(n,a,b,r)\ \hbox{(all other cases)} & (1,0) & $[n,a,b,r]\in\gL_n^*$ \\
&&\\[-0.3cm]
G_\cT(\cT,\cT) & (1,0) &  \\
G_\cT(\cT,Q_8) & (1,0) &  \\
G_\cT(L_{12}^\cT,C_2) & (1,0),\ (1,{j-k\over\sqrt{2}}) & two, $\cong_\HH G_\cO(L_{14}^\cO,1)$  \\
G_\cT(L_{12}^\cT,1) & (1,\ga {j-k\over\sqrt{2}}),\ -1<\ga\le1 & infinite family, $\cong_\HH G_{12}$ \\
&&\\[-0.3cm]
G_\cO(\cO,\cO) & (1,0) &  \\
G_\cO(\cO,\cT) & (1,0) &  \\
G_\cO(L_{32}^\cO,Q_8) & (1,0) &  \\
G_\cO(L_{20}^\cO,C_2) & (1,0),\ (1,{j-k\over\sqrt{2}}) & two systems \\
G_\cO(L_{18}^\cO,1) & (1,\ga {j-k\over\sqrt{2}}),\ -1<\ga\le1 & infinite family, $\cong_\HH G_{13}$  \\
G_\cO(L_{14}^\cO,1) & (1,0),\ (1,{j-k\over\sqrt{2}}) & two, $\cong_\HH G_\cT(L_{12}^\cT,C_2) $  \\
&&\\[-0.3cm]
G_\cI(\cI,\cI) & (1,0) &  \\
	G_\cI(L_{32}^\cI,C_2) & (1,0),\ (1,{j-\tau i-\gs k\over2})  & two systems \\
G_\cI(L_{30}^\cI,1) & (1,\ga{j-\tau i-\gs k\over2}),\ -1<\ga\le1  & infinite family, $\cong_\HH G_{22}$  \\
G_\cI(L_{20}^\cI,C_2) & (1,0) &  \\
G_\cI(L_{20}^\cI,1) & (1,0) &  \\ [3pt]
\hline
\end{tabular}
\end{center}
\end{table}

\section{Concluding remarks}

We have calculated all the quaternionic systems of imprimitivity for the rank two real, 
complex and quaternionic reflection groups, see Tables \ref{RealSystems-table},
\ref{ComplexSystems-table}, \ref{QuaternionicSystems-table}), respectively.

The systems of primitivity for the quaternionic reflection groups of rank two were
calculated in \cite{T25} (Theorems 6.4, 6.5, 6.6). These results are not directly 
comparable with our results (see Example \ref{TGnabr}), 
as ``copies'' of the reflection groups 
which do not satisfy the inclusions (see Figure \ref{TOIinclusionsfigure}) that we used.
In this regard, see Table \ref{WalTayTOI}.

\begin{table}
\caption{The imprimitive quaternionic reflection groups for $K=\cT,\cO,\cI$ that have 
	additional systems of imprimitivity (left table, six groups up to conjugacy) and those that don't (right table), 
	and the corresponding groups used in \cite{T25}.}
\label{WalTayTOI}
\begin{tabular}{ |  >{$}l<{$} | >{$}l<{$} | >{$}l<{$} | }
\hline
&&\\[-0.3cm]
	G & \hbox{$G$ of \cite{T25}} & |G| \\[0.1cm]
\hline
&&\\[-0.3cm]
G_\cT(L_{12}^\cT,C_2) & G(\cT,C_2,\rho(\gd)) & 96 \\
G_\cT(L_{12}^\cT,1) & G(\cT,1,\rho(\gd)) & 48 \\
G_\cO(L_{20}^\cO,C_2) & G(\cO,C_2,1) & 192 \\
G_\cO(L_{14}^\cO,1) & G(\cO,1,\gb) & 96 \\
G_\cO(L_{18}^\cO,1) & G(\cO,1,\rho(\gd)) & 96 \\
G_\cI(L_{32}^\cI,C_2) & G(\cI,C_2,1) & 480 \\
G_\cI(L_{30}^\cI,1) & G(\cI,1,\rho(j)) & 240 \\
\hline
\end{tabular}
\quad
\begin{tabular}{ |  >{$}l<{$} | >{$}l<{$} | >{$}l<{$} | }
\hline
&&\\[-0.3cm]
        G & \hbox{$G$ of \cite{T25}} & |G| \\[0.1cm]
\hline
&&\\[-0.3cm]
G_\cT(\cT,\cT) & G(\cT,\cT,1) & 1152 \\
G_\cT(\cT,Q_8) & G(\cT,\cD_2,\rho(\gd)) & 384 \\
G_\cO(\cO,\cO) & G(\cO,\cO,1) & 4068 \\
G_\cO(\cO,\cT) & G(\cO,\cT,1) & 2304 \\
G_\cO(L_{32}^\cO,Q_8) & G(\cO,\cD_2,1) & 768 \\
G_\cI(\cI,\cI) & G(\cI,\cI,1) & 28800 \\
G_\cI(L_{20}^\cI,C_2) & G(\cI,C_2,\Theta) & 480 \\
G_\cI(L_{20}^\cI,1) & G(\cI,1,\Theta) & 240 \\
\hline
\end{tabular}
\end{table}

From Table \ref{QuaternionicSystems-table}, we can observe the following.

\begin{corollary} 
\label{refsordertwocor}
An imprimitive quaternionic reflection group of rank two can have more than
one system of imprimitivity only if $H=1,C_2$, i.e., every reflection 
	has order two. 
This is not a sufficient condition, e.g., $G_\cI(L_{20}^\cI,1)$ and 
$G_\cI(L_{20}^\cI,C_2)$ have just one system of imprimitivity.
The number of systems of imprimitivity 
of an imprimitive quaternionic reflection group of rank two
	can be one (infinitely many cases),
	two (three cases), three (three infinite families), five (one case), or
	infinite (three cases).
\end{corollary}

There are also groups for $\cD_n$ 
with just one system of imprimitivity, e.g.,
$G(12,3,4,1)$ and $G(12,2,3,2)$.

The first observation of Corollary \ref{refsordertwocor}
 is the Corollary 6.3 of \cite{T25}.
With hindsight, this can be proved directly, by showing
that reflections of order $m\ge 3$ are incompatible with
multiple systems of imprimitivity.

\begin{lemma} 
\label{insightlemma}
If a group $G\subset M_2(\FF)$ contains a reflection of order $m\ge3$,
then it has at most one system of imprimitivity.
\end{lemma}

\begin{proof}
Suppose, without loss of generality, that $G$ is imprimitive, and so,
after an appropriate conjugation, contains
$$ H:=\cG(\{\},\{h\})=\inpro{\pmat{h&0\cr 0&1}}, $$
where $h$ has order $m\ge3$.
Then, by Lemma \ref{imprimitivitylemma},
for $H$ to have an additional system of imprimitivity given by $(1,q)$, 
we must have $h=-|q|^2$, which is not possible.
\end{proof}

In view of the fact that the real/complex systems  of imprimitivity 
for a rank two real/complex group depends only on the associated collineation group
(see observation 1 in the introduction), Lemma \ref{insightlemma} suggests
that multiple systems of imprimitivity are unusual. Indeed, there is just one
example in each case.

\begin{example} 
\label{multisetsrealcomplex}
There are unique rank two real and complex collineation 
groups which have multiple sets of imprimitivity.
They are $G(2,1,2)$ and $G(4,2,2)$, with systems of imprimitivity given by
	$(1,0)$, $(1,1)$ and $(1,0)$, $(1,1)$, $(1,i)$, respectively.
\end{example}

A rank two quaternionic reflection group with five systems of imprimitivity 
was observed in \cite{BW25}. This is in fact the maximal finite number
of such systems possible.

\begin{example}
\label{KMUBexample}
The three rank two primitive quaternionic reflection groups of type $P$ 
with primitive complexifications given in \cite{C80} have 
the following irreducible imprimitive normal subgroup
$$ K=G(2,1,1,2)=G_{Q_8}(Q_8,C_2)
	=\cG(\{1,i,j,k\},\{\})
	=\cG(\{1,i,j,k\},\{-1\})
	. $$
In \cite{BW25}, it was observed that this group has $10$ reflections
corresponding to the five pairs of mutually unbiased bases
\begin{equation*}
\bigl\{\pmat{1\cr0},\pmat{0\cr1}\bigr\}, \quad 
\bigl\{{1\over\sqrt{2}}\pmat{1\cr\pm1}\bigr\}, \quad
\bigl\{{1\over\sqrt{2}}\pmat{1\cr\pm i}\bigr\}, \quad
\bigl\{{1\over\sqrt{2}}\pmat{1\cr\pm j}\bigr\}, \quad 
\bigl\{{1\over\sqrt{2}}\pmat{1\cr\pm k}\bigr\},
\end{equation*}
which are the systems of imprimitivity for $K$. By Corollary \ref{refsordertwocor},
this is the maximum finite number of systems of imprimitivity for a rank two 
quaternionic reflection group, and the only time that it occurs.

It was also observed that $G(4,1,2,2)$ has three systems of imprimitivity.
This first group in the family $G(n,1,{n\over2},2)$, $n\ge4$, with 
exactly three systems of imprimitivity.
\end{example}

The calculation of systems of imprimitivity can be used to see whether 
a reflection group is imprimitive, or conjugate to another with a different
sized reflection system (see Example \ref{sysimpconjex}).
In this regard, \cite{T25} calculated the systems of imprimitivity
for the infinite families of rank two primitive reflection groups with
nonmonomial imprimitive complexifications of \cite{C80} (Lemma 3.3), 
and determined that three of them are in fact {\em imprimitive}, namely

$$ 
\mathcal C_4\boxdot\mathcal O \cong_\HH G_\cO(L_{20}^\cO,C_2), \quad
\mathcal C_4\boxdot_2\mathcal O \cong_\HH G_\cO(L_{14}^\cO,1) \cong_\HH G_\cT(L_{12}^\cT,C_2), \quad
\mathcal C_4\boxdot\mathcal I \cong_\HH G_\cI(L_{32}^\cI,C_2).
$$





We have repeatedly used our observation that larger groups have fewer systems 
of imprimitivity -- both to prove and understand results. 
We give one final example.

\begin{example}
The reflection system for $G(n,a,b,r)$ contains the ${2n\over a}$ root of unity $\go^a$,
so that we have
$$ G({2n\over a},{2n\over a},2)\subset G({2n\over a},{2n\over a r},2)\subset G(n,a,b,r). $$
Hence, by Theorem \ref{imprimiitivecomplexsys}, 
the possible quaternionic systems of imprimitivity for $G(n,a,b,r)$
are given by the matrices $U$ of (\ref{zjmechanics}). 
One can 
check whether $[g]=U^{-1}gU$ is monomial,
for each generator $g$ of $G(n,a,b,r)$,
to obtain Theorem \ref{G(n,a,b,r)sys}.

The particular case $n=2$ relates the groups of Example \ref{multisetsrealcomplex}
and Example \ref{KMUBexample}
$$ G(4,4,2)\subset G(4,2,2)\subset K= G(2,1,1,2), $$
with Theorem \ref{imprimiitivecomplexsys} suggesting that $(1,1)$, $(1,i)$ may give systems 
of imprimitivity for $K$. 

It appears that the map $G(n,a,b,r)\mapsto G({2n\over a},{2n\over a r},2)$ is one-to-one.
\end{example}





\bibliographystyle{alpha}
\bibliography{references}
\nocite{*}

\end{document}